\documentclass[draft,12pt]{amsart}
\usepackage{amsfonts}
\usepackage{amssymb}
\numberwithin{equation}{section}
\theoremstyle{plain}
 \newtheorem{thm}{Theorem}[section]
 \newtheorem{lem}[thm]{Lemma}
 \newtheorem{cor}[thm]{Corollary}
 \newtheorem{prop}[thm]{Proposition}

\theoremstyle{definition}

\newcommand{\al}{\alpha}%\allow
\newcommand{\bt}{\beta}
\newcommand{\gm}{\gamma}

\newcommand{\dl}{\delta}

\newcommand{\ep}{\varepsilon}

\newcommand{\sg}{\sigma}

\newcommand{\ph}{\varphi}%\phantom

\newcommand{\rh}{\rho}

\newcommand{\dar}{\downarrow}

\newcommand{\mcal}{\mathcal}
\newcommand{\mrm}{\mathrm}

\newcommand{\eqd}{\overset{\mathrm d}{=}}
\newcommand{\wh}{\widehat}

\newcommand{\R}{\mathbb{R}}
\newcommand{\N}{\mathbb{N}}
\newcommand{\Z}{\mathbb{Z}}
\newcommand{\Q}{\mathbb{Q}}

\newcommand{\law}{\mathcal L}

\newcommand{\sym}{\,\mathrm{sym}}

\begin{document}
\setlength{\baselineskip}{18pt}
\setlength{\parindent}{1.8pc}
\allowdisplaybreaks
\title{Properties of stationary distributions\\
of a sequence of generalized\\
Ornstein--Uhlenbeck processes}
\author{Alexander Lindner and Ken-iti Sato}
\begin{abstract}
The infinite (in both directions) sequence of the distributions
$\mu^{(k)}$ of the stochastic integrals $\int_0^{\infty-}
c^{-N_{t-}^{(k)}} dL_t^{(k)}$ for integers $k$ is investigated. Here
$c>1$ and $(N_t^{(k)},L_t^{(k)})$, $t\geq0$, is a bivariate compound
Poisson process with L\'evy measure concentrated on three points
$(1,0)$, $(0,1)$, $(1,c^{-k})$.  The amounts of the normalized
L\'evy measure at these points are denoted by $p$, $q$, $r$. For
$k=0$ the process $(N_t^{(0)},L_t^{(0)})$ is marginally Poisson and
$\mu^{(0)}$ has been studied by Lindner and Sato (Ann.\ Probab.\
{\bf37} (2009), 250--274). The distributions $\mu^{(k)}$ are the
stationary distributions of a sequence of generalized
Ornstein--Uhlenbeck processes structurally related in some way.
Continuity properties of $\mu^{(k)}$ are shown to be the same as
those of $\mu^{(0)}$.  The dependence on $k$ of infinite
divisibility of $\mu^{(k)}$ is clarified. The problem to find
necessary and sufficient conditions in terms of $c$, $p$, $q$, and
$r$ for $\mu^{(k)}$ to be infinitely divisible is somewhat involved,
but completely solved for every integer $k$. The conditions depend
on arithmetical properties of $c$. The symmetrizations of
$\mu^{(k)}$ are also studied. The distributions $\mu^{(k)}$ and
their symmetrizations are $c^{-1}$-decomposable, and it is shown
that, for each $k\neq 0$, $\mu^{(k)}$ and its symmetrization may be
infinitely divisible without the corresponding factor in the
$c^{-1}$-decomposability relation being infinitely divisible. This
phenomenon was first observed by Niedbalska-Rajba (Colloq.\ Math.\
{\bf 44} (1981), 347--358) in an artificial example. The notion of
quasi-infinite divisibility is introduced and utilized, and it is
shown that a quasi-infinitely divisible distribution on $[0,\infty)$
can have its quasi-L\'evy measure concentrated on $(-\infty,0)$.
\end{abstract}
\maketitle

\section{Introduction}

Let   $\{ V_t, t\geq 0\}$ be a generalized Ornstein--Uhlenbeck
process associated with a bivariate L\'evy process $\{ (\xi_t,
\eta_t), t\geq0\}$ with initial condition $S$. That is, $\{ V_t \}$
is a stochastic process defined by
\begin{equation} \label{1-1}
V_t = e^{-\xi_t} \left( S + \int_0^t e^{\xi_{s-}} \, d\eta_s
\right),
\end{equation}
where $\{ (\xi_t, \eta_t)\}$ and $S$ are assumed to be independent
(Carmona et al.\ \cite{CPY97,CPY01}). Define two other bivariate
L\'evy process $\{(\xi_t,L_t)\}, t\geq 0\}$ and $\{ (U_t,
L_t),t\geq0\}$ by
\begin{equation} \label{1-2}
\left( \begin{array}{c} U_t \\ L_t  \end{array} \right) = \left(
\begin{array}{l} \xi_t - \sum_{0< s \leq t} \left(e^{- (\xi_s -
\xi_{s-})}
-1 + (\xi_{s}-\xi_{s-})\right) - t \,2^{-1}\alpha_{\xi,\xi} \\
 \eta_t + \sum_{0 < s \leq t} (e^{-(\xi_s - \xi_{s-})}-1)
(\eta_s - \eta_{s-}) - t\,\al_{\xi,\eta}  \end{array}\right)
\end{equation}
where $\al_{\xi,\xi}$ and $\al_{\xi,\eta}$ are the $(1,1)$ and the
$(1,2)$ element  of the Gaussian covariance matrix of $\{ (\xi_t,
\eta_t)\}$, respectively.
%Then \eqref{1-2a} means that
%$\{e^{-\xi_t}\}$ is the Dol\'eans-Dade exponential of $\{U_t\}$, and
%it is easy to see from \eqref{1-2a} and \eqref{1-2} that
%\begin{equation} \label{1-3a}
%\{(U_t,L_t), t\geq 0\} = \{ (U_t, \eta_t + [U,\eta]_t),t\geq 0\},
%\end{equation}
%where $\{[U,\eta]_t, t\geq 0\}$ denotes the quadratic covariation of
%$\{U_t\}$ and $\{\eta_t\}$.
Then $\{V_t, t\geq 0\}$ is the unique solution of the stochastic
differential equation
\begin{equation} \label{eq-SDE}
dV_t = -V_{t-} \, dU_t + dL_t, \quad t\geq 0, \quad V_0 = S,
\end{equation}
the filtration being such that $\{V_t\}$ is adapted and $\{U_t\}$
and $\{L_t\}$ are both semimartingales with respect to it (see
Maller et.~al~\cite{MMS}, p.~428, or Protter~\cite{Pr}, Exercise V.27).
Hence we shall also refer to a generalized Ornstein--Uhlenbeck
process associated with $\{(\xi_t,\eta_t)\}$ as the solution of the
SDE \eqref{eq-SDE} {\it driven by $\{(U_t,L_t)\}$}.
% We
%refer to Protter~\cite{Pr} for the concepts used from stochastic
%integration theory.
Let
\begin{equation}\label{1-4}
 \mu =\mathcal{L} \left(\int_0^{\infty-}
e^{- \xi_{s-} } \, dL_s\right),
\end{equation}
whenever the improper integral exists, where $\mcal L$ stands for
\lq\lq distribution of". If $\{\xi_t\}$ drifts to $+\infty$ as
$t\to\infty$ (or, alternatively, under a minor non-degeneracy
condition), a necessary and sufficient condition for $\{ V_t \}$ to
be a strictly stationary process under an appropriate choice of $S$
%being independent of $\{(\xi_t,\eta_t)\}$
is the almost sure
convergence of the improper integral in \eqref{1-4}; in this case
$\mu$ is the unique stationary marginal distribution (Lindner and
Maller \cite{LM}).
 The condition for the convergence of the
improper integral in \eqref{1-4} in terms of the L\'evy--Khintchine
triplet of $\{( \xi_t, L_t)\}$ is given by Erickson and Maller
\cite{EM}. Properties of the distribution $\mu$ are largely unknown,
apart from some special cases. For example, it  is selfdecomposable
if $\eta_t=t$ and $\xi_t= (\log c) N_t$ for a Poisson process
$\{N_t\}$ and a constant $c>1$ (Bertoin et al.~\cite{BBY}), or if
$\{\xi_t\}$ is spectrally negative and drifts to $+\infty$  as
$t\to\infty$ (see Bertoin et al.~\cite{BLM} and
Kondo et al.\ \cite{KMS} for a multivariate
generalization). Bertoin
 et al.\ \cite{BLM} have shown that the distribution in \eqref{1-4}
is always continuous unless degenerated to a Dirac measure. In
Lindner and Sato \cite{LS}, the distribution $\mu$ in \eqref{1-4}
and its symmetrization was studied for the case when
$\{(\xi_t,L_t)\} = \{((\log c) N_t, L_t)\}$ for a constant $c>1$
and a  bivariate L\'evy process $\{ (N_t,L_t)\}$ such that both
$\{N_t\}$ and $\{L_t\}$ are Poisson process; the L\'evy measure of
$\{(N_t,L_t)\}$ is then concentrated on the three points $(1,0)$,
$(0,1)$ and $(1,1)$.

In this paper we extend the setup of our paper \cite{LS}, by
defining a sequence of bivariate L\'evy processes
$\{(N_t^{(k)},L_t^{(k)}), t\geq0\}$,
$k\in\Z=\{\ldots,-1,0,1,\ldots\}$, in the following way. The process
$\{(N_t^{(k)},L_t^{(k)})\}$ has the characteristic function
\begin{equation}\label{1-5}
E[e^{i(z_1N_t^{(k)}+z_2L_t^{(k)})}]=\exp\left[
t\int_{\R^2}(e^{i(z_1x_1+z_2x_2)}-1) \nu^{(k)}(dx)\right],\quad
(z_1,z_2)\in\R^2,
\end{equation}
where the L\'evy measure $\nu^{(k)}$ is concentrated on at most
three points $(1,0)$, $(0,1)$, $(1,c^{-k})$ with $c>1$ and
\[
u= \nu^{(k)}( \{ (1,0) \}), \quad v= \nu^{(k)} ( \{ (0,1) \}), \quad
w= \nu^{(k)} ( \{ (1,c^{-k}) \} ).
\]
We assume that $u+w>0$ and $v+w>0$, so that $\{N_t^{(k)}\}$ is a
Poisson process with parameter $u+w$ and $\{L_t^{(k)}\}$ is a
compound Poisson process with L\'evy measure  concentrated on at
most two points $1$, $c^{-k}$ with total mass $v+w$. In particular,
$\{L_t^{(0)}\}$ is a Poisson process with parameter $v+w$. We define
the normalized L\'evy measure, which has mass
\[
p=\frac{u}{u+v+w},\quad q=\frac{v}{u+v+w},\quad r=\frac{w}{u+v+w}
\]
at the three points. We have  $p,q,r\ge0$ and $p+q+r=1$. The
assumption that $u+w>0$ and $v+w>0$ is now written as $p+r>0$ and
$q+r>0$.  We are interested in continuity properties and conditions
for infinite divisibility of the distribution
\begin{equation}\label{1-6}
\mu^{(k)}=\mcal L \left(\int_0^{\infty-} c^{-N^{(k)}_{s-}}
dL_s^{(k)}\right), \quad k\in\Z.
\end{equation}
For $k=0$ the distribution $\mu^{(0)}$ is identical with the
distribution $\mu_{c,q,r}$ studied in our paper \cite{LS}. As will
be shown in Proposition~\ref{p2a} below,
 $\mu^{(k+1)}$ is the unique stationary
distribution of the generalized Ornstein-Uhlenbeck process
associated with $\{ ((\log c) N_t^{(k)}, L_t^{(k)})\}$ as defined in
\eqref{1-1}, while $\mu^{(k)}$
 appears naturally as the unique stationary
distribution of the SDE~\eqref{eq-SDE} driven by $\{((1-c^{-1})
N_t^{(k)}, L_t^{(k)})\}$. The fact that $\{(N_t^{(k)}, L_t^{(k)})\}$
is related to both $\mu^{(k+1)}$ and $\mu^{(k)}$ in a natural way
explains the initial interest in the distributions $\mu^{(k)}$ with
general $k\in \Z$. As discussed below, they have some surprising
properties which cannot be observed for $k=0$.

%In \cite{LS} it is shown that infinite divisibility of $\mu^{(0)}$
%does not depend on $c$, and a criterion for the infinite
%divisibility is given in terms of the parameters $q$ and $r$.  As
%for continuity properties, $\mu^{(0)}$ is a Dirac measure if and
%only if $r=1$; if $r<1$, then $\mu_{c,q,r}$ is either
%continuous-singular or absolutely continuous. Classification into
%these two cases in terms of the parameters is an unsolved problem.
%It heavily depends on algebraic properties of $c$. Sufficient
%conditions for the continuous-singularity and sufficient conditions
%for the absolute continuity are given there.

In contrast to the situation in our paper~\cite{LS}, where
$\mu^{(0)}$ was studied, continuity properties of $\mu^{(k)}$ are
easy to handle, in the sense that they are reduced to those of
$\mu^{(0)}$; but classification of $\mu^{(k)}$ into infinitely divisible and
non-infinitely divisible cases is more complicated than that
of $\mu^{(0)}$. We will give  a complete answer to this
problem.
% In particular, unlike for $k\neq 0$, infinite divisibility
% of $\mu^{(k)}$ for $k\neq 0$ depends also on the parameter $c$.
%, and
%for $k<0$ the dependence on the parameters is more involved than in
%the case $k>0$.
The criterion for infinite divisibility of $\mu^{(k)}$ depends on
arithmetical properties of $c$.  It is more involved for $k<0$
than for $k>0$.  If $k<0$ and $c^j$ is an integer for some positive
integer $j$, we will have to introduce a new class of functions
$h_{\al,\gm}(x)$ with integer parameters $\al\geq2$ and $\gm\geq1$
to express the criterion.  In the case that $k<0$ and $c^j$ is
not an integer for any positive integer $j$, the hardest situation
is where $c^j=3/2$ for some integer $j$.  In this situation,
however, we will express the criterion by 149 explicit
inequalities between $p$, $q$, and $r$.
It will be also shown that, for $p$, $q$, and $r$ fixed,
the infinite divisibility
of $\mu^{(k)}$, $k\in \Z$, has the following monotonicity: if
$\mu^{(k)}$ is infinitely divisible for some $k=k_0$, then
$\mu^{(k)}$ is infinitely divisible for all $k\geq k_0$. Further,
% as will follow from our characterizations,
if $p>0$ and $r>0$, then
$\mu^{(k)}$ is non-infinitely divisible for all $k$ sufficiently
close to $-\infty$. The case where $\mu^{(k)}$ is non-infinitely
divisible for all $k\in\Z$ is also characterized in terms of the
parameters.

The investigation of the law $\mu^{(k)}$ is related to the study of
$c^{-1}$-decomposable distributions. For $b\in(0,1)$ a distribution
$\sg$ on $\R$ is said to be $b$-decomposable if there is a
distribution $\rh$ such that
\[
\wh\sg(z)=\wh\rh(z)\,\wh\sg(bz),\quad z\in\R.
\]
 Here $\wh\sg(z)$ and $\wh\rh(z)$ denote the characteristic
functions of $\sg$ and $\rh$. The ``factor'' $\rho$ is not
necessarily uniquely determined by $\sg$ and $b$, but it is if
$\wh{\sg}(z) \neq 0$ for $z$ from a dense subset of $\R$. If $\rh$
is infinitely divisible, then so is $\sg$, but the converse is not
necessarily true as pointed out by Niedbalska-Rajba \cite{Ni} in a
somewhat artificial example. The study of $b$-decomposable
distributions is made by Lo\`eve \cite{L45},
Grincevi\v{c}jus~\cite{Gr}, Wolfe \cite{Wo}, Bunge \cite{Bu},
Watanabe \cite{Wa00}, and others. In particular, any
$b$-decomposable distribution which is not a Dirac measure is either
continuous-singular or absolutely continuous (\cite{Gr} or
\cite{Wo}).

We will show that, for $k\in\Z$, $\mu^{(k)}$ is $c^{-1}$-decomposable
and explicitly give the distribution $\rh^{(k)}$ satisfying
\begin{equation}\label{1-8}
\wh\mu^{(k)}(z)=\wh\rh^{(k)}(z)\,\wh\mu^{(k)}(c^{-1}z),
\end{equation}
where $\wh\mu^{(k)}(z)$ and $\wh\rh^{(k)}(z)$ are the characteristic
functions of $\mu^{(k)}$ and $\rh^{(k)}$. The distribution
$\rh^{(k)}$ is unique here as will follow from Proposition~\ref{p2c}
below. A criterion for infinite divisibility of $\rh^{(k)}$ for
$k\in\Z$ in terms of $c$, $p$, $q$, and $r$ will be given;
it is simpler than that of $\mu^{(k)}$.
% for all $k\in \Z$.
%For  infinite divisibility
%of $\mu^{(k)}$, a criterion is given for $k\geq0$, and  under some
%restriction also for $k<0$.
In particular, it will be shown that for
every $k\neq 0$ there are parameters $c, p,q,r$ such that the factor
$\rh^{(k)}$ is not infinitely divisible while $\mu^{(k)}$ is
infinitely divisible. This is different from the situation $k=0$
treated in \cite{LS}, since such a phenomenon does not happen for
$\mu^{(0)}$. Allowing $k\neq 0$,  we obtain a lot of examples
satisfying this phenomenon, and unlike in
 Niedbalska-Rajba \cite{Ni}, our examples
are connected with simple stochastic processes.

We also consider the symmetrizations $\mu^{(k)\,\mrm{sym}}$ for
general $k\in \Z$. Then  $\mu^{(k)\,\mrm{sym}}$ is again
$c^{-1}$-decomposable and satisfies
\begin{equation}\label{1-9}
\wh\mu^{(k)\sym}(z)=\wh\rh^{(k)\sym}(z)\,\wh\mu^{(k)\sym}(c^{-1}z).
\end{equation}
Necessary and sufficient conditions for infinite divisibility of
$\mu^{(k)\,\mrm{sym}}$ and of $\rho^{(k)\,\mrm{sym}}$ are obtained. In
particular, it will be shown that if $k\neq 0$, then
$\mu^{(k)\,\mrm{sym}}$ can be infinitely divisible without
$\rh^{(k)\,\mrm{sym}}$ being infinitely divisible, a phenomenon
which does not occur for
% which is observed for symmetric $b$-decomposable distributions for
% the first time, and again does not hold in the situation
$\mu^{(0)}$ treated in \cite{LS}.
The argument we use to characterize infinite
divisibility of $\mu^{(k)\,\mrm{sym}}$ for $k\in \Z$ is new also in
the situation $k=0$, and simplifies the proof given in \cite{LS} for
that situation considerably.

We introduce the following notion for distributions having
L\'evy--Khintchine-like representation.
% \begin{defn}\label{d1a}
A distribution $\sg$ on $\R$ is called {\it quasi-infinitely
divisible} if
\begin{equation}\label{d1a1}
\wh\sg(z)=\exp\left[ i \gm z - az^2+\int_{\R}
(e^{izx}-1-izx1_{[-1,1]}(x))\,\nu_{\sg} (dx)\right],
\end{equation}
where $\gamma, a \in\R$ and $\nu_{\sg}$ is a signed measure on $\R$
with total variation measure $|\nu_{\sg}|$ satisfying
$\nu_{\sg}(\{0\})=0$ and $\int_{\R}(x^2\land1)\,|\nu_{\sg}|(dx)
<\infty$. The signed measure $\nu_\sg$ will be called {\it
quasi-L\'evy measure} of $\sg$.
% \end{defn}
Note that $\gm$, $a$ and $\nu_{\sg}$ in \eqref{d1a1} are unique if
they exist. Infinitely divisible distributions on $\R$ are
quasi-infinitely divisible. A quasi-infinitely divisible
distribution $\sg$ on $\R$ is infinitely divisible if and only if
$a\geq 0$ and the negative part of $\nu_{\sg}$ in the Jordan
decomposition is zero. See E12.2 and E12.3 of \cite{Sa}. We shall
see in Corollary~\ref{cor-4.2} that some of the distributions
$\mu^{(k)}$, supported on
$\R_+=[0,\infty)$,  are quasi-infinitely divisible with non-trivial
quasi-L\'evy measure being concentrated
on $(-\infty,0)$. Such a phenomenon does not occur in
infinitely divisible case.

In this paper $ID$, $ID^0$, and $ID^{00}$ respectively denote the
class of infinitely divisible distributions on $\R$, the class of
quasi-infinitely divisible, non-infinitely divisible distributions
on $\R$, and the class of distributions on $\R$ which are not
quasi-infinitely divisible. When characterizing infinite
divisibility of $\rho^{(k)}$, $\mu^{(k)}$, $\rho^{(k)\,\rm sym}$ and
$\mu^{(k)\,\rm sym}$ we shall more precisely determine to which of
the classes $ID$, $ID^0$ and $ID^{00}$ the corresponding
distributions belong.

Without the name of quasi-infinitely
divisible distributions, the property that $\sg$ satisfies
\eqref{d1a1} with $\nu_{\sg}$ having non-trivial negative part is known
to be useful in showing that $\sg$ is not  infinitely divisible,
in books and papers such as Gnedenko
and Kolmogorov \cite{GK} (p.\,81), Linnik and Ostrovskii \cite{LO}
(Chap.~6, \S\,7) and Niedbalska-Rajba \cite{Ni}.
We single out the class $ID^0$ for two reasons.  The first is that
$\mu$ in $ID^0$ has a manageable characteristic function, which is
the quotient of two  infinitely divisible  characteristic functions.
The second is that the notion is useful in studying the
symmetrization $\mu^{\rm sym}$ of $\mu$.  Already in
Gnedenko and Kolmogorov \cite{GK} p.\,82 an example of $\mu\not\in ID$
satisfying $\mu^{\rm sym}\in ID$ is given in this way.  It is noticed in
\cite{LS} that $\mu^{(0)\,\mrm{sym}}$ (or $\rh^{(0)\,\mrm{sym}}$) can be
in $ID$ without $\mu^{(0)}$ (or $\rh^{(0)}$) being in $ID$.
We will show the same phenomenon occurs also for
$\mu^{(k)}$ and $\rh^{(k)}$.

The paper is organized as follows: in Section 2 we describe the
$c^{-1}$-decomposa\-bi\-li\-ty of $\mu^{(k)}$, $k\in\Z$, and its
consequences.  Section 3 deals with continuity properties of
$\mu^{(k)}$, $k\in\Z$.  In Sections 4, 5, and 6 results on infinite
divisibility and quasi-infinite divisibility of $\rh^{(k)}$ and
$\mu^{(k)}$ are given for general $k$, positive $k$, and negative
$k$, respectively. The last Section 7 discusses the symmetrizations.

We shall assume throughout the paper that $c>1$, $p+r>0$ and $q+r>0$
without further mentioning. The following notation will be used.
$\N$ (resp.\ $\N_0$) is the set of positive (resp.\ nonnegative)
integers. $\N_{\mrm{even}}$ (resp.\ $\N_{\mrm{odd}}$) is the set of
even (resp.\ odd) positive integers. The Lebesgue measure of $B$ is
denoted by $\mrm{Leb}\,(B)$. The dimension of a measure $\sg$,
written $\dim\,(\sg)$, is the infimum of $\dim\,B$, the Hausdorff
dimension of $B$,  over all Borel sets $B$ having full $\sg$
measure. $H(\rh)$ is the entropy of a discrete measure $\rh$. $\mcal
B(\R)$ is the class of Borel sets in $\R$. The Dirac measure at a
point $x$ is denoted by $\dl_x$.

%It will be shown that the L\'evy process
%$\{(\xi_t^{(k)},\et_t^{(k)}),t\geq0\}$ is the $L$-process of the
%L\'evy process $\{(\xi_t^{(k-1)},\et_t^{(k-1)}),t\geq0\}$, and that
%the generalized Ornstein--Uhlenbeck process associated with $\{
%(\xi_t^{(k)}, \eta_t^{(k)})\}$ has a unique stationary distribution
%$\mu^{(k)}$. The set-up given above with the assumptions $p+r>0$ and
%$q+r>0$ is maintained throughout this paper. Our main purpose is to
%study continuity properties and conditions for infinite divisibility
%for the distributions $\mu^{(k)}$, $k\in\Z$.
%
%We call the process $\{ (\xi_t, L_t)\}$ the $L$-process of $\{
%(\xi_t, \eta_t)\}$.  It follows from \eqref{1-2} that
%\begin{equation*}\label{1-3}
%\eta_t = L_t + \sum_{0 < s \leq t} (e^{\xi_s - \xi_{s-}} -1) ( L_s -
%L_{s-})+ t\,\al.
%\end{equation*}

%%%%%%%%%%%%%%%%%%%%%%%%%%%%%%%%%%%%%%%%%%%%%
%%%%%%%%%%%%%%%%%%%%%%%%%%%%%%%%%%%%%%%%%%%%%Section2
\section{The $c^{-1}$-decomposability and its consequences}

We start with the following proposition which clarifies the
relations between $\{(N^{(k)}_t, L^{(k)}_t)\}$, $\{(N^{(k-1)}_t,
L_t^{(k-1)})\}$ and $\mu^{(k)}$.

\begin{prop}\label{p2a}
Let $c,p,q,r$ be fixed and let $k\in\Z$.  Then
\begin{equation*} \label{eq-relation-sequence} \{(N_t^{(k)},L_t^{(k)})\}
%\stackrel{d}{=}
\eqd \{ (N_t^{(k-1)}, L_t^{(k-1)} + \sum_{0<s\leq t}
(e^{-\log (c) (N_s^{(k-1)} - N_{s-}^{(k-1)})}-1) (L_s^{(k-1)} -
L_{s-}^{(k-1)}))\},
\end{equation*}
so that $\{((1-c^{-1})N_t^{(k)}, L_t^{(k)})\}$ is equal in
distribution to the right-hand-side of \eqref{1-2} when applied with
$\{(\xi_t,\eta_t)\} = \{(\log(c) N_t^{(k-1)}, L_t^{(k-1)})\}$. The
integral $\int_0^{\infty-} c^{-N_{s-}^{(k)}} \, dL_s^{(k)}$ exists
as an almost sure limit, and its distribution $\mu^{(k)}$ is the
unique stationary distribution of the generalized
Ornstein--Uhlenbeck process associated with $\{((\log c)
N_t^{(k-1)},\linebreak L_t^{(k-1)})\}$ as defined in \eqref{1-1},
equivalently $\mu^{(k)}$ is the unique stationary distribution of
the SDE \eqref{eq-SDE} driven by $\{ ((1-c^{-1})N_t^{(k)},
L_t^{(k)})\}$.
\end{prop}

\begin{proof}
The process $\{(N_t^{(k-1)},L_t^{(k-1)})\}$ is a bivariate compound
Poisson process. Its jump size is determined by the normalized
L\'evy measure and for $k\in \Z$ we have
\begin{align*}
\{(N_t^{(k)},L_t^{(k)})\} & \eqd \{ ( N_t^{(k-1)} ,
\sum_{0<s\le t}
c^{-(N_s^{(k-1)}-N_{s-}^{(k-1)})} (L_s^{(k-1)}-L_{s-}^{(k-1)}))\} \\
& = \{ ( N_t^{(k-1)} , L_t^{(k-1)} + \sum_{0<s\le t}
(c^{-(N_s^{(k-1)}-N_{s-}^{(k-1)})}-1)
(L_s^{(k-1)}-L_{s-}^{(k-1)}))\},
\end{align*}
% \begin{eqnarray*}
% \{(N_t^{(k)},L_t^{(k)})\} & \eqd & \left\{ \left( N_t^{(k-1)} ,
% \sum_{0<s\le t}
% c^{-(N_s^{(k-1)}-N_{s-}^{(k-1)})} (L_s^{(k-1)}-L_{s-}^{(k-1)})\right)\right\} \\
% & = &  \left\{ \left( N_t^{(k-1)} , L_t^{(k-1)} + \sum_{0<s\le t}
% (c^{-(N_s^{(k-1)}-N_{s-}^{(k-1)})}-1)
% (L_s^{(k-1)}-L_{s-}^{(k-1)})\right)\right\},
% \end{eqnarray*}
giving the first relation. The existence of the improper stochastic
integral follows from the law of large numbers. The remaining
assertions are then clear from the discussion in the introduction,
where \eqref{1-4} was identified as the unique stationary
distribution of the corresponding stochastic process.
\end{proof}

Let $T$ be the first jump time for $\{N_t^{(k)}\}$ and let
\begin{equation}\label{2-1}
\rh^{(k)}=\mcal L(L_T^{(k)}).
\end{equation}

\begin{prop}\label{p2b}
For $k\in\Z$ the distribution $\mu^{(k)}$ is $c^{-1}$-decomposable and satisfies
\eqref{1-8}.
The characteristic function of $\mu^{(k)}$ has expression
\begin{equation}\label{p2b2}
\wh\mu^{(k)}(z)=\prod_{n=0}^{\infty}\wh\rh^{(k)}(c^{-n}z),\quad z\in\R.
\end{equation}
\end{prop}

\begin{proof}
By the strong Markov property for L\'evy processes we have
\begin{align*}
 \int_0^{\infty-} c^{-N^{(k)}_{s-}} dL_s^{(k)} &= L_T^{(k)} + c^{-1}
\int_{T+}^{\infty-} c^{-(N^{(k)}_{s-} - N_T^{(k)})} \,
d(L_{\cdot}^{(k)} -
L_T^{(k)})_s\\
&\eqd %\stackrel{d}{=}
L_T^{(k)}+c^{-1}\int_0^{\infty-}
c^{-N_{s-}^{(k)\prime}} dL_s^{(k)\prime},
\end{align*}
where $\{(N_t^{(k)\prime},L_t^{(k)\prime})\}$ is an independent copy
of $\{(N_t^{(k)},L_t^{(k)})\}$. This shows \eqref{1-8} and hence
$\mu^{(k)}$ is $c^{-1}$-decomposable. Since \eqref{1-8} implies
\[
\wh\mu^{(k)}(z)=\wh\mu^{(k)}(c^{-l}z)
\prod_{n=0}^{l-1}\wh\rh^{(k)}(c^{-n}z),\qquad z\in \R,\; l\in \N,
\]
we obtain \eqref{p2b2}.
\end{proof}

\begin{prop}\label{p2c}
For $k\in\Z$ the distributions  $\rh^{(k)}$ and $\mu^{(k)}$ satisfy the following.
\begin{gather}
\rh^{(k)}=\sum_{m=0}^{\infty} q^m p\,\dl_m+ \sum_{m=0}^{\infty} q^m
r\,\dl_{m+c^{-k}} ,\label{p2c1}\\
\wh\rh^{(k)}(z)=\frac{p+re^{ic^{-k} z}}{1-qe^{iz}} ,\label{p2c2}\\
\wh\mu^{(k)}(z)=\prod_{n=0}^{\infty} \frac{p+re^{ic^{-k-n}z}}{1-qe^{ic^{-n}z}} ,
\label{p2c3}\\
\wh\mu^{(k)}(z)=\wh\mu^{(k+1)}(z)\left(\frac{p}{p+r}+\frac{r}{p+r}e^{ic^{-k}z}
\right),
\label{p2c4}\\
\mu^{(k)}(B)=\frac{p}{p+r}\mu^{(k+1)}(B)+\frac{r}{p+r}\mu^{(k+1)}(B-c^{-k}),\quad
B\in\mcal B(\R),\label{p2c5}\\
\wh\mu^{(k+1)}(z)=\wh\mu^{(k)}(c^{-1}z)\frac{1-q}{1-qe^{iz}}.\label{p2c6}
\end{gather}
Further, the distribution $\rho^{(k)}$ is uniquely determined by
$\mu^{(k)}$ and \eqref{1-8}.
\end{prop}

\begin{proof}
Let $S_1, S_2,\ldots$ be the successive jump sizes of the compound
Poisson process $\{(N_t^{(k)},L_t^{(k)})\}$. Then
\begin{align*}
\rh^{(k)}&=P[S_1=(1,0)]\,\dl_0+P[S_1=(1,c^{-k})]\,\dl_{c^{-k}}\\
&\quad+\sum_{m=1}^{\infty} P[S_1=(0,1),\ldots,S_m=(0,1), S_{m+1}=(1,0)]\,\dl_m\\
&\quad+\sum_{m=1}^{\infty} P[S_1=(0,1),\ldots,S_m=(0,1),
S_{m+1}=(1,c^{-k}) ]\,\dl_{m+c^{-k}},
\end{align*}
which is equal to the right-hand side of \eqref{p2c1}. Note that $q=1-(p+r)<1$.
It follows from \eqref{p2c1} that
\[
\wh\rh^{(k)}(z)=\sum_{m=0}^{\infty} q^m pe^{imz}+ \sum_{m=0}^{\infty} q^m
re^{i(m+c^{-k})z} ,
\]
which is written to \eqref{p2c2}. This, combined with \eqref{p2b2}, gives \eqref{p2c3}.
It follows from \eqref{p2c3} that
\begin{align*}
\wh\mu^{(k)}(z)&=\lim_{l\to\infty}\prod_{n=0}^l
\frac{p+re^{ic^{-k-n}z}}{1-qe^{ic^{-n}z}}\\
&=\lim_{l\to\infty}\frac{p+re^{ic^{-k}z}}{p+re^{ic^{-k-1-l}z}}
\prod_{n=0}^l \frac{p+re^{ic^{-k-1-n}z}}{1-qe^{ic^{-n}z}}\\
&=\frac{p+re^{ic^{-k}z}}{p+r}
\wh\mu^{(k+1)}(z).
\end{align*}
This is \eqref{p2c4}.  It means that $\mu^{(k)}$ is a mixture of $\mu^{(k+1)}$
with the translation of $\mu^{(k+1)}$ by $c^{-k}$, as in \eqref{p2c5}.
Similarly,
\begin{align*}
\wh\mu^{(k+1)}(z)&=\lim_{l\to\infty}\prod_{n=0}^l
\frac{p+re^{ic^{-k-1-n}z}}{1-qe^{ic^{-n}z}}\\
&=\lim_{l\to\infty}\frac{1-qe^{ic^{-l-1}z}}{1-qe^{iz}}
\prod_{n=0}^l \frac{p+re^{ic^{-k-1-n}z}}{1-qe^{ic^{-n-1}z}}\\
&=\frac{1-q}{1-qe^{iz}} \wh\mu^{(k)}(c^{-1}z),
\end{align*}
which is \eqref{p2c6}. Finally, since
$$\sum_{n=0}^\infty \left| \frac{p+r e^{ic^{-k-n}z}}{1-q
e^{ic^{-n}z}}-1\right| \leq \sum_{n=0}^\infty \frac{r
|e^{ic^{-k-n}z}-1| + q |e^{ic^{-n}z}-1|}{1-q} < \infty,$$ the
infinite product in \eqref{p2c3} cannot be zero unless $p+r
e^{ic^{-k-n}z}=0$ for some $n\in \N_0$. It follows that
$\wh\mu^{(k)}(z) \neq 0$ for $z$ from a dense subset of $\R$, so
that $\rho^{(k)}$ is uniquely determined by $\mu^{(k)}$ and
\eqref{1-8}.
\end{proof}

%%%%%%%%%%%%%%%%%%%%%%%%%%%%%%%%%%%%%%%%%%
%%%%%%%%%%%%%%%%%%%%%%%%%%%%%%%%%%%%%%%%%%Section 3
\section{Continuity properties for all $k$}

Continuity properties for $\mu^{(k)}$ do not depend on $k$, as the
following theorem shows.
As a consequence of Proposition \ref{p2c}, $\mu^{(k)}$ is a Dirac measure
if and only if $r=1$.  If $r<1$, then $\mu^{(k)}$ is
either continuous-singular or absolutely continuous, since it is
$c^{-1}$-decomposable.

\begin{thm}\label{t3a}
Let $c,p,q,r$ be fixed and let $k\in\Z$. Then:

{\rm(i)}  $\mu^{(k)}$ is absolutely continuous if and only if $\mu^{(0)}$
is absolutely continuous.

{\rm(ii)}  $\mu^{(k)}$ is continuous-singular if and only if $\mu^{(0)}$
is continuous-singular.

{\rm(iii)}\quad $\dim\,(\mu^{(k)})=\dim\,(\mu^{(0)})$.
\end{thm}

\begin{proof}
It is enough to show that absolute continuity,
continuous-singularity, and the dimension of $\mu^{(k)}$ do not
depend on $k$. We use \eqref{p2c5}.

(i) If $p=0$, then $\mu^{(k)}$ is a translation of $\mu^{(k+1)}$ and
the assertion is obvious. Assume that $p>0$. Let $\mu^{(k+1)}$ be
absolutely continuous.  If $B$ is a Borel set with $\mrm{Leb}
(B)=0$, then $\mu^{(k+1)}(B)=0$, $\mrm{Leb} (B-c^{-k})=0$, and
$\mu^{(k+1)}(B-c^{-k})=0$ and hence $\mu^{(k)}(B)=0$ from
\eqref{p2c5}. Hence $\mu^{(k)}$ is absolutely continuous.
Conversely, let $\mu^{(k)}$ be absolutely continuous. If $B$ is a
Borel set with
 $\mrm{Leb} (B)=0$, then $\mu^{(k)}(B)=0$ and hence $\mu^{(k+1)}(B)=0$ from
 \eqref{p2c5} and from $p>0$.  Hence $\mu^{(k+1)}$ is absolutely continuous.

(ii) We know that $\mu^{(k)}$ is a Dirac measure if and only if $\mu^{(0)}$
is.  Hence (ii) is equivalent to (i).

(iii)  We may assume $p>0$.  Let $d^{(k)}=\dim\,(\mu^{(k)})$.
For any $\ep>0$ there is a Borel set $B$ such
that $\mu^{(k+1)}(B)=1$ and $\dim\,B <d^{(k+1)}+\ep$.
Since
\begin{align*}
\mu^{(k)}(B\cup(B+c^{-k}))&=\frac{p}{p+r}\mu^{(k+1)}(B\cup(B+c^{-k}))+
\frac{r}{p+r}\mu^{(k+1)}((B-c^{-k})\cup B)\\
&\geq\frac{p}{p+r}\mu^{(k+1)}(B)+\frac{r}{p+r}\mu^{(k+1)}(B)=1,
\end{align*}
we have  $\mu^{(k)}(B\cup(B+c^{-k}))=1$.
Since $\dim\,(B\cup(B+c^{-k}))=\dim\,B$, this shows $d^{(k)}\leq d^{(k+1)}$.
On the other hand, for any $\ep>0$ there is a Borel set $E$ such
that $\mu^{(k)}(E)=1$ and $\dim\,E <d^{(k)}+\ep$. If $\mu^{(k+1)}(E)<1$, then
\[
\mu^{(k)}(E)< \frac{p}{p+r}+\frac{r}{p+r}\mu^{(k+1)}(E-c^{-k})\leq 1,
\]
a contradiction. Hence $\mu^{(k+1)}(E)=1$
and $d^{(k+1)}\leq d^{(k)}$.
\end{proof}

  By virtue of Theorem \ref{t3a}, all results
on continuity properties of $\mu^{(0)}$ in \cite{LS} are applicable
to $\mu^{(k)}$, $k\in\Z$.  Thus, by the method of Erd\H{o}s \cite{Er},
$\mu^{(k)}$
is continuous-singular if $c$ is a Pisot--Vijayaraghavan number
and $q>0$ (see
the survey \cite{PSS} on this class of numbers).
On the other hand, for almost all $c$ in $(1,\infty)$, sufficient
conditions for absolute continuity of $\mu^{(k)}$ are given by an
essential use of results of Watanabe \cite{Wa00} (see \cite{LS}).

Recall that  for any discrete probability measure $\sg$ the entropy
$H(\sg)$ is defined by
\[
H(\sg)=-\sum_{x\in C} \sg(\{x\})\log \sg(\{x\}),
\]
where $C$ is the set of points of positive $\sg$ measure.

\begin{thm}\label{t3b}
Let $c,p,q,r$ be fixed and let $k\in\Z$.  We have
\begin{equation}\label{t3b1}
\dim\,(\mu^{(k)})\leq H(\rh^{(k)})/\log c
\end{equation}
and
\begin{equation}\label{t3b2}
H(\rh^{(k)})\leq H(\rh^{(1)}).
\end{equation}
More precisely,
\begin{equation}\label{t3b3}
H(\rh^{(k)})\begin{cases} =H(\rh^{(1)})\qquad &\text{if\/ $k>0$},\\
= H(\rh^{(1)})\qquad &\text{if\/ $k< 0$ and $c^{-k}\not\in\N$},\\
< H(\rh^{(1)})\qquad &\text{if\/ $k\le 0$, $c^{-k}\in\N$, and
$p,q,r>0$}.
\end{cases}
\end{equation}
\end{thm}

\begin{proof}
The inequality \eqref{t3b1} follows from Theorem 2.2 of Watanabe \cite{Wa00}.
 If $k>0$ or
if $k\le 0$ and $c^{-k}\not\in\N$, then, in the expression
\eqref{p2c1} of $\rh^{(k)}$, all $k$ and $k+c^{-k}$ for $k\in\N_0$
are distinct points and hence $H(\rh^{(k)})$ does not depend on $k$.
For general $k\in\Z$, some of the points $k$ and $k+c^{-k}$ for
$k\in\N_0$ may coincide, which makes the entropy smaller than or
equal to $H(\rh^{(1)})$. This proves \eqref{t3b2}.  If $k\le 0$,
$c^{-k}\in\N$, and $p,q,r>0$, then some of points with positive mass
indeed amalgamate and the entropy becomes smaller than
$H(\rh^{(1)})$.
\end{proof}

A straightforward calculus gives
\begin{equation}\label{3-1}
H(\rh^{(1)})=(-p\log p-q\log q-r\log r)/(1-q),
\end{equation}
with the interpretation $x\log x=0$ for $x=0$.

\begin{thm}\label{t3c}
Let $r<1$.  If\/ $\log c>(\log 3)/(1-q)$, then $\mu^{(k)}$
is continuous-singular for all $k\in\Z$.
\end{thm}

\begin{proof}
It follows from \eqref{3-1} that $H(\rh^{(1)})\leq (\log 3)/(1-q)$. Hence
by Theorem \ref{t3b} $\dim\,(\mu^{(k)})<1$ if $\log c>(\log 3)/(1-q)$.
\end{proof}

%%%%%%%%%%%%%%%%%%%%%%%%%%%%%%%%%%%%%%%%%%
%%%%%%%%%%%%%%%%%%%%%%%%%%%%%%%%%%%%%%%%%%Section 4
\section{General results on infinite divisibility for all $k$}

We give two theorems concerning the classification of $\rh^{(k)}$
and $\mu^{(k)}$, $k\in\Z$, into $ID$, $ID^0$, and $ID^{00}$. The
first theorem concerns $\rh^{(k)}$ and $\mu^{(k)}$, while the second
deals with $\mu^{(k)}$. We also obtain examples of
quasi-infinitely divisible distributions on $\R_+$
with quasi-L\'evy measure
being concentrated on $(-\infty,0)$.

\begin{thm}\label{t4a}
{\rm (i)}  If $p=0$ or if $r=0$, then $\rh^{(k)}$ and $\mu^{(k)}$
are in $ID$ for every $k\in\Z$.

{\rm (ii)}   If\/ $0<r<p$, then  $\rh^{(k)}$ and $\mu^{(k)}$ are in
$ID\cup ID^0$ for every $k\in\Z$, with quasi-L\'evy measures being
concentrated on $(0,\infty)$.

{\rm (iii)}  If\/ $0<p<r$, then $\rh^{(k)}$ and $\mu^{(k)}$ are in
$ID^{0}$ for every $k\in\Z$.

{\rm (iv)}  If\/ $p=r>0$, then $\rh^{(k)}$ and $\mu^{(k)}$ are in
$ID^{00}$ for every $k\in\Z$.
\end{thm}

It is noteworthy that in this theorem the results do
not depend on $k$ and the results for $\rh^{(k)}$ and $\mu^{(k)}$ are the same.
By virtue of this theorem, in the classification of $\rh^{(k)}$ and $\mu^{(k)}$,
it remains only to find, in the case $0<r<p$, necessary and sufficient
conditions for their infinite divisibility.

\begin{proof}[Proof of Theorem \ref{t4a}]
If $r=0$, then \eqref{p2c1} shows that $\rh^{(k)}$ does not depend
on $k$, is a geometric distribution, and hence in $ID$, which
implies that $\mu^{(k)}$ does not depend on $k$ and is in $ID$. If
$p=0$, then \eqref{p2c1} shows that $\rh^{(k)}$ is a shifted
geometric distribution, and hence in $ID$, implying that
$\mu^{(k)}\in ID$. Hence (i) is true.

Let us prove (ii). Assume that $0<r<p$.
It follows from \eqref{p2c2} that
\begin{align*}
\wh\rh^{(k)}(z)
&=\exp\left[ -\log(1-qe^{iz})+\log(1+(r/p)e^{ic^{-k}z}) +\log p\right]\\
&=\exp\left[ \sum_{m=1}^{\infty}\left( m^{-1}q^m e^{imz}-m^{-1}(-r/p)^m
e^{imc^{-k}z}\right) +\log p \right]\\
&=\exp\left[ \sum_{m=1}^{\infty}\left( m^{-1}q^m (e^{imz}-1)-m^{-1}(-r/p)^m
(e^{imc^{-k}z}-1)\right) +\mrm{const} \right]
\end{align*}
and, letting $z=0$, we see that the constant is zero. Hence
\begin{equation}\label{t4b1}
\wh\rh^{(k)}(z)=\exp\left[\int_{(0,\infty)} (e^{izx}-1)\nu_{\rh^{(k)}}(dx)
\right],
\end{equation}
where $\nu_{\rh^{(k)}}$ is a signed measure
given by
\begin{equation}\label{t4b2}
\nu_{\rho^{(k)}} = \sum_{m=1}^\infty \left[ m^{-1} q^m \delta_m
+ (-1)^{m+1} m^{-1} (r/p)^m \delta_{m c^{-k}}\right]
\end{equation}
with finite total variation.
Then it follows from \eqref{p2b2} that
\begin{equation}\label{t4b3}
\wh\mu^{(k)}(z)=\exp\left[\int_{(0,\infty)} (e^{izx}-1)\nu_{\mu^{(k)}}(dx)
\right]
\end{equation}
with
\begin{equation}\label{t4b4}
\nu_{\mu^{(k)}} = \sum_{n=0}^\infty \sum_{m=1}^\infty \left[ m^{-1} q^m
\delta_{mc^{-n}} + (-1)^{m+1}m^{-1} (r/p)^m \delta_{m c^{-k-n}}\right].
\end{equation}
Notice that $\int_{(0,\infty)} x\left| \nu_{\mu^{(k)}}\right|
(dx)<\infty$. Hence $\rh^{(k)}$ and $\mu^{(k)}$ are in $ID\cup ID^0$
for every $k\in\Z$, and the quasi-L\'evy measures are concentrated
on $(0,\infty)$ by \eqref{t4b2} and \eqref{t4b4}.

To prove (iii), assume $0<p<r$. Then by \eqref{p2c2} and a
calculation similar to the one which lead to \eqref{t4b1}
\begin{align*}
\wh\rh^{(k)}(z) %&=\frac{p+re^{ic^{-k} z}}{1-qe^{iz}}\\
&=\frac{1+(p/r)e^{-ic^{-k} z}}{1-qe^{iz}} re^{ic^{-k} z}\\
&=\exp\,[-\log(1-qe^{iz}) +\log(1+(p/r)e^{-ic^{-k} z}) +\log r+ic^{-k} z ]\\
&=\exp\left[\sum_{m=1}^{\infty} m^{-1}q^me^{imz}
-\sum_{m=1}^{\infty} m^{-1}(-p/r)^m e^{-imc^{-k} z} +\log r+ic^{-k}z
\right].
\end{align*}
Thus
\begin{equation}
%&=\exp\left[\sum_{m=1}^{\infty} m^{-1}q^m(e^{imz}-1)
%-\sum_{m=1}^{\infty} m^{-1}(-p/r)^m (e^{-imc^{-k} z}-1) +ic^{-k} z
%\right]\\
\wh\rh^{(k)}(z) =\exp\left[\int_{\R}(e^{izx}-1)\nu_{\rh^{(k)}}(dx)+ic^{-k} z
\right], \label{quas-1}
\end{equation}
where
\begin{equation} \label{quas-2}
\nu_{\rh^{(k)}}=\sum_{m=1}^{\infty} m^{-1}q^m \dl_m
+\sum_{m=1}^{\infty} (-1)^{m+1} m^{-1}(p/r)^m \dl_{-mc^{-k}}.
\end{equation}
Clearly the negative part in the Jordan decomposition of
$\nu_{\rh^{(k)}}$ is non-zero.  Hence $\rh^{(k)}\in ID^0$. As in
(ii), this together with \eqref{p2b2} implies
\begin{align}
\wh\mu^{(k)}(z)  %&= \prod_{n=0}^{\infty}\wh\rh^{(k)}(c^{-n} z)\\
%&=\exp\left[\sum_{n=0}^{\infty} \sum_{m=1}^{\infty} \left(
%m^{-1}q^m(e^{ic^{-n}mz}-1)+(-1)^{m+1} m^{-1}(p/r)^m (e^{-imc^{-k-n}
%z}-1)
%\right)\right.\\
%&\qquad\qquad\left. +i\sum_{n=0}^{\infty} c^{-k-n} z \right]\\
&=\exp\left[\int_{\R}(e^{izx}-1)\nu_{\mu^{(k)}}(dx)+i\sum_{n=0}^\infty
c^{-k-n} z \right] \label{quas-3}
\end{align}
with
\begin{equation} \label{quas-4}
\nu_{\mu^{(k)}}=\sum_{n=0}^{\infty} \sum_{m=1}^{\infty} \left(
m^{-1}q^m \dl_{c^{-n}m} +(-1)^{m+1} m^{-1}(p/r)^m
\dl_{-mc^{-k-n}}\right).
\end{equation}
Again we have $\int_{\R} |x|\,|\nu_{\mu^{(k)}}|(dx)<\infty$. Hence
$\mu^{(k)}\in ID\cup ID^0$.  If  $\mu^{(k)}\in ID$, then not only
$\nu_{\mu^{(k)}}$ is non-negative but also $\nu_{\mu^{(k)}}$ is
concentrated on $(0,\infty)$, since $\mu^{(k)}$ is concentrated on
$\R_+$.  However
\begin{align*}
&\int_{(-\infty,0)} |x|\,\nu_{\mu^{(k)}}(dx)=
\sum_{n=0}^{\infty} \sum_{m=1}^{\infty} mc^{-k-n}(-1)^{m+1} m^{-1}(p/r)^m\\
&\qquad =\sum_{n=0}^{\infty} c^{-k-n} \sum_{m=1}^{\infty}
(-1)(-p/r)^m =\frac{c^{-k}}{1-c^{-1}} \frac{p/r}{1+p/r}\neq 0.
\end{align*}
Thus $\nu_{\mu^{(k)}}$ is not concentrated on $(0,\infty)$.  It
follows that $\mu^{(k)}\not\in ID$.

To show (iv), observe that if $p=r>0$, then $\wh\rh^{(k)}(z)=0$ and
$\wh\mu^{(k)}(z)=0$ for $z=c^k \pi$ from \eqref{p2c2} and
\eqref{1-8}, which implies that $\rh^{(k)}$ and $\mu^{(k)}$ are in
$ID^{00}$.
\end{proof}

The assertion (iii) is new even in the case $k=0$. In Theorem~2.2
of~\cite{LS} it was only shown that $\mu^{(0)}\not\in ID$ if $0<p<r$
by using the representation $e^{-\varphi(\theta)}$ for the Laplace
transform of infinitely divisible distributions on $\R_+$ with
$\varphi'(\theta)$ being completely monotone. The present proof of
Theorem~\ref{t4a}~(iii) is simpler and shows even that $\mu^{(0)}
\in ID^0$.

It is worth noting that in contrast to infinitely divisible
distributions, whose L\'evy measure must be concentrated on
$(0,\infty)$ if the distribution itself is concentrated on $\R_+$,
the proof of the previous Theorem shows that the same conclusion
does not hold for quasi-infinitely divisible distributions. Even
more surprising, the quasi-L\'evy measure of a
quasi-infinitely divisible distribution on $\R_+$
can be concentrated on
$(-\infty,0)$.

\begin{cor} \label{cor-4.2}
Let $q=0$, $0<p<r$ and $k\in \Z$. Then $\rho^{(k)}$ and $\mu^{(k)}$
both have bounded support contained in $\R_+$, but the quasi-L\'evy
measures $\nu_{\rho^{(k)}}$ and $\nu_{\mu^{(k)}}$ in the
L\'evy--Khintchine-like representations \eqref{quas-1} and
\eqref{quas-3}, respectively, are concentrated on $(-\infty, 0)$.
\end{cor}

\begin{proof}
Since $q=0$, $\rho^{(k)}$  has distribution supported on two points
$0$, $c^{-k}$ and $\mu^{(k)}$ is a scaled infinite Bernoulli
convolution, which is supported on $[0,c^{1-k}/(c-1)]$.  The
assertion on the quasi-L\'evy measures is immediate from
\eqref{quas-1} -- \eqref{quas-4}.
\end{proof}

% {\bf Comment by Alex: shall we insert the following comment
% somehow:}
Actually, in the Gnedenko--Kolmogorov book \cite{GK}, the example in
p.\,81 gives, after shifting by $+1$, the distribution obtained by
deleting some mass at the point $0$ from a geometric distribution
and by normalizing.  It coincides with $\rh^{(0)}$ for $0<p<r$
and $q>0$.
% {\bf ? Maybe it can be skipped?}

\medskip
\noindent{\it Remark}.  Consider $\sg=p\dl_0+r\dl_1$ with $0<p<1$
and $r=1-p$.  If $p=r=1/2$, then $\sg\in ID^{00}$.  If $p>r$, then
$\sg\in ID^0$ with quasi-L\'evy measure supported on $\N$. If $p<r$,
then $\sg\in ID^0$ with quasi-L\'evy measure supported on $-\N$.
Indeed, $\sg=\rh^{(0)}$ with $q=0$ and the proof of Theorem
\ref{t4a} shows this fact.  By scaling and shifting, we see that any
distribution supported on two points in $\R$ has similar properties.
% Among distributions on $\R_+$ with
%  bounded support, some are quasi-infinitely
% divisible and some are not.  Indeed, distributions $\rh^{(k)}$ and
% $\mu^{(k)}$ with $p,r>0=q$ belong to $ID^0$ if $p\neq r$, and belong
% to $ID^{00}$ if $p= r$. This follows from Theorem~\ref{t4a} and the
% same reasoning as in the proof of Corollary~\ref{cor-4.2} to show
% that $\rho^{(k)}$ and $\mu^{(k)}$ have bounded support contained in
% $\R_+$ without being Dirac measures.
\medskip

We have the following monotonicity property of $\mu^{(k)}$ in $k$.

\begin{thm}\label{t4b}
Let $k\in\Z$ and the parameters $c, p,q,r$ be fixed.
If $\mu^{(k)}\in ID$, then $\mu^{(k+1)}\in ID$.
% {\rm (ii)} If $\mu^{(k)}\in ID\cup ID^0$, then $\mu^{(k+1)}\in
% ID\cup ID^0$.
\end{thm}

\begin{proof}
Recall the relation \eqref{p2c6} in Proposition \ref{p2c}.  The
factor $(1-q)(1-qe^{iz})^{-1}$ is the characteristic function of a
geometric distribution if $q>0$ and of $\delta_0$ if $q=0$, both of
which are infinitely divisible. Hence the proof is straightforward.
\end{proof}

%%%%%%%%%%%%%%%%%%%%%%%%%%%%%%%%%%%%%%%%%%%
%%%%%%%%%%%%%%%%%%%%%%%%%%%%%%%%%%%%%%%%%%%Section 5
\section{Conditions for infinite divisibility for $k>0$}

In the classification of $\rh^{(k)}$ and
$\mu^{(k)}$ into $ID$, $ID^0$, and $ID^{00}$, it remains only to find
a necessary and sufficient condition for infinite divisibility
in the case $0<r<p$ (see Theorem \ref{t4a}).

The results in \cite{LS} show the
following for $\rh^{(0)}$ and $\mu^{(0)}$.

\begin{prop}\label{p5a} Assume $0<r<p$.
{\rm(i)} If $r\leq pq$, then $\rh^{(0)},\,\mu^{(0)}\in ID$.

{\rm(ii)} If $r>pq$,  then $\rh^{(0)},\,\mu^{(0)}\in ID^0$.
\end{prop}

We stress that
$\rh^{(0)}\in ID$ and $\mu^{(0)}\in ID$ are equivalent and that the
classification does not depend on $c$.

For $k$ a nonzero integer, to find the infinite divisibility condition is harder.
The condition depends on $c$,
and $\mu^{(k)}\in ID$ does not necessarily imply $\rh^{(k)}\in ID$.

\begin{thm}\label{t5a}
Let $k\in\N$.  Assume $0<r<p$. Then $\rh^{(k)}\in ID$ if
and only if $c^k=2$ and $r^2\leq p^2q$.
\end{thm}

\begin{proof}
For later use in the proof of Theorem~\ref{t6a} allow for the moment
that $k\in \Z$.  We have shown the expression \eqref{t4b1} for
$\wh\rh^{(k)}(z)$ with the signed measure $\nu_{\rh^{(k)}}$ of
\eqref{t4b2}. We have $\rh^{(k)}\in ID$ if and only if
$\nu_{\rh^{(k)}}\geq 0$.  If $s:= 2c^{-k} \not\in \N$, then
$\rh^{(k)}\in ID^0$, since $\nu_{\rho^{(k)}} (\{2c^{-k}\}) = -
2^{-1} (r/p)^2 < 0$. If $2 c^{-k}=s \in \N$, then
\begin{align*}
\rh^{(k)}\in ID &\quad\Leftrightarrow\quad \sum_{m=1}^\infty\left[
 m^{-1} q^m \delta_m + m^{-1} (-1)^{m+1} (r/p)^m \delta_{sm/2}\right]
\geq0\\
%& \quad\Leftrightarrow\quad \sum_{m=1}^\infty m^{-1} q^m \delta_m -
%\sum_{m=1}^\infty (2m)^{-1} (r/p)^{2m} \delta_{m} \geq0\\
& \quad\Leftrightarrow\quad (sm)^{-1} q^{sm} \geq (2m)^{-1}
(r/p)^{2m}, \quad
\forall\; m\in\N,
\end{align*}
which is equivalent to
\begin{equation}
q^s \geq (s/2)^{1/m} (r/p)^2, \quad
\forall\; m\in\N .\label{eq-for-rho}
\end{equation}
Now assume that $k>0$. Then necessarily $s=1$ and \eqref{eq-for-rho}
is equivalent to $q \geq (r/p)^2$. Thus the proof is complete.
\end{proof}

\begin{thm}\label{t5b}
Let $k\in\N$. Assume $0<r<p$. Then $\mu^{(k)}\in ID$ if
and only if one of the following holds: {\rm(a)} $r\leq pq$;
{\rm(b)} $c^l=2$ for some $l\in\{ 1,2,\ldots,k\}$ and $r^2\leq p^2q$.
\end{thm}

\begin{proof}
Keep in mind the assumption $0<r<p$. If $q=0$, then $\mu^{(k)}\in
ID^0$ since it has compact support without being a Dirac measure as
shown in the proof of Corollary~\ref{cor-4.2}, hence cannot be in
$ID$. If (a) holds, then $\mu^{(k)}\in ID$ by Theorem~\ref{t4b},
since $\mu^{(0)}\in ID$. If
% $r^2\leq p^2 q$ and $c^l=2$ for some $l\in\{ 1,2,\ldots,k\}$,
(b) holds, then $\mu^{(k)}\in ID$ by
Theorem~\ref{t4b}, since $\mu^{(l)}\in ID$ by Theorem \ref{t5a}. In
view of these facts and Theorem~\ref{t4a}, in order to prove our
theorem, it is enough to
show the following two facts:\\
(A)  {\it If $q>0$, $r>pq$, and $c^l\neq 2$ for all $l\in\{ 1,2,\ldots,k\}$,
then $\mu^{(k)}\in ID^0$.}\\
(B)  {\it If $q>0$ and $r^2>p^2 q$, then $\mu^{(k)}\in ID^0$.}

Suppose that $q>0$ and  $r>pq$. We have the expression  \eqref{t4b3}
of $\wh\mu^{(k)}(z)$ with the signed measure $\nu_{\mu^{(k)}}$ of
\eqref{t4b4}. Hence
\begin{equation}\label{t5b1}
\nu_{\mu^{(k)}}= \sum_{n=k}^\infty \sum_{m=1}^\infty a_m \delta_{c^{-n}m} +
\sum_{n=0}^{k-1} \sum_{m=1}^\infty m^{-1} q^m \delta_{c^{-n}m},
\end{equation}
with
\begin{equation}\label{t5b2}
a_m = m^{-1} (q^m  -( - r/p)^m) .
\end{equation}
Observe that
\begin{equation}\label{t5b3}
a_m < 0,\qquad\forall\,m\in\N_{\mrm{even}}.
\end{equation}
In order to prove (A), assume further that $c^l\neq 2$ for all
$l\in\{ 1,2,\ldots,k\}$.
If $c^l$ is irrational for all $l\in\N$, then $c^{-n}m\neq c^{-n'}m'$
whenever $(m,n)\neq(m',n')$, and we see that the negative part of
$\nu_{\mu^{(k)}}$ is nonzero, and hence $\mu^{(k)}\in ID^0$.
So suppose that $c^l$ is rational for some $l\in\N$ and let $l_0$ be the
smallest such $l$. Denote $c^{l_0} = \al/\bt$ with $\alpha,\beta\in\N$
having no common divisor. As in Case 2 in the proof of Theorem 2.2 (b) of
\cite{LS}, let $f$ be the largest
$t\in\N_0$ such that $2^t$ divides $\beta$. Let $m\in\N_{\mrm{even}}$ and
denote
\begin{align*}
G_m &= \{ (n',m') \in \N_0 \times \N\colon c^{-n'} m' =
m,\;\text{$m'$ odd}\},\\
H_m &= \{ (n',m') \in \N_0 \times \N\colon c^{-n'} m' =
m,\;\text{$m'$ even}\},\\
G_m^{(k)} &= \{ (n',m') \in \{k,k+1,\ldots\} \times \N\colon
 c^{-n'} m' =c^{-k}m,\;\text{$m'$ odd}\},\\
H_m^{(k)} &= \{ (n',m') \in \{k,k+1,\ldots\} \times \N\colon
 c^{-n'} m' =c^{-k}m,\;\text{$m'$ even}\}.
\end{align*}
Then $(n',m') \in G_m^{(k)}$ if and only if $(n'-k,m') \in G_m$, and
the same is true for $H_m^{(k)}$ and $H_m$. %In particular,
%$G_m^{(k)}$ contains at most one element for each $m$ as shown in
%\cite{LS}.
Now if $m\in \N_{\mrm{even}}$ is such that $c^{n-k}m\not\in \N$ for
all $n\in \{0,\ldots, k-1\}$ (by assumption this is satisfied in
particular for $m=2$), then
\begin{equation}\label{eq-t5b3a}
\begin{split}
\nu_{\mu^{(k)}}(\{c^{-k} m\})&= \sum_{n=k}^\infty \sum_{m'=1}^\infty
a_{m'} \delta_{c^{-n}m'} ( \{ c^{-k} m\}) =
\sum_{(n',m')\in G_m^{(k)} \cup
H_m^{(k)}} a_{m'}\\
&\leq a_m + \sum_{(n',m')\in G_m} a_{m'}.
\end{split}
\end{equation}
If $G_2$ is empty, then \eqref{eq-t5b3a} gives $\nu_{\mu^{(k)}} (\{
2c^{-k}\})\leq a_2<0$, showing that $\mu^{(k)}\in ID^0$. So suppose
that $G_2$ is non-empty. As shown in \cite{LS}, this implies $f\geq
1$ and hence $\beta$ is even, hence $\alpha$ odd, and
$\alpha\not=1$, since $c>1$. Now choose $m=m_j = 2^{jf}$ for
$j\in\N$. Then for each $j$, as shown in \cite{LS},  $G_{m_j}$
contains at most one element, and if $G_{m_j}\neq \emptyset$ its
unique element
 $(n',m')$ is given by
$n'=jl_0$ and $m'=m_j' = c^{jl_0} m_j$. For $j\in\N$ such that
\begin{equation}\label{t5b4}
c^{n-k} m_j\not\in \N\quad\text{for all }n\in\{0,1,\ldots, k-1\},
\end{equation}
\eqref{eq-t5b3a} gives
\[
\nu_{\mu^{(k)}} (\{ c^{-k} m_j\}) \leq \begin{cases} a_{m_j} < 0 , &
G_{m_j} = \emptyset, \\
a_{m_j} + a_{m_j'}, & G_{m_j} \neq \emptyset.
\end{cases}
\]  Since $a_{m_j} + a_{m_j'}<0$ for $j\in\N$ large enough such that $G_{m_j}
\neq \emptyset$ (see~\cite{LS}, p.261),  we obtain $\mu^{(k)}\in
ID^0$, provided that, for large enough $j$, condition \eqref{t5b4}
holds. If $c, c^2, \ldots, c^k$ are all irrational, then
\eqref{t5b4} is clear. If some of them are rational, then $l_0\leq
k$ and $c^{l_0} = \alpha/\beta \not= 2$ and
\[
c^{-s l_0} m_j = \beta^s 2^{jf}/\alpha^s \not\in \N, \quad \forall\;
s\in \N
\]
(since $\alpha>1$ odd), and we have \eqref{t5b4}, recalling
that $c^{-n} m_j$ is irrational if $n$ is not an integer
multiple of $l_0$.  This finishes the proof of (A).

Let us prove the statement (B).  Since $r^2>p^2 q$ implies $r>pq$,
we assume that $q>0$ and $r>pq$.  Then we have \eqref{t5b1},
\eqref{t5b2}, and \eqref{t5b3}.  Since we have already proved (A),
we consider only the case where $c^l=2$ for some $l\in\{
1,2,\ldots,k\}$. This $l$ is unique. Then it is easy to see that
$c^{l'}$ is irrational for $l'=1,\ldots,l-1$. Let $s$ be the largest
non-negative integer satisfying $l s \leq k-1$. Let
$m\in\N_{\mrm{odd}}$. We have
\begin{align*}
\nu_{\mu^{(k)}} ( \{ c^{-l s} m\})
&=\sum_{n=k}^\infty \sum_{m'=1}^\infty a_{m'} \delta_{c^{-n}m'}(\{2^{-s}m\}) +
\sum_{n=0}^{k-1} \sum_{m'=1}^\infty (m')^{-1} q^{m'} \delta_{c^{-n}m'}
(\{2^{-s}m\})\\
&=S_1+S_2.
\end{align*}
Recall that $m$ is odd and that $c^{l'}$ is irrational for
$l'=1,\ldots,l-1$. Then we see that $S_1=\sum_{n=1}^{\infty} a_{2^n m}$
and $S_2=m^{-1} q^m$.  Hence we obtain
\[
\nu_{\mu^{(k)}} ( \{ c^{-l s} m\})<a_{2m}+m^{-1} q^m=
(2m)^{-1}(q^{2m}-(r/p)^{2m})+m^{-1} q^m
\]
from \eqref{t5b3}.
We conclude that if $\mu^{(k)}\in ID$, then
\[
0<1+ q^m/2 - (r^2/(p^2q))^m /2,\qquad\forall\,m\in\N_{\mrm{odd}},
\]
which implies $r^2\leq p^2q$. This finishes the proof of (B).
\end{proof}

The following corollary is now immediate from Theorems~\ref{t5a}
and~\ref{t5b}.

\begin{cor}\label{c5a}
If $k\in\N$, then parameters $c,p,q,r$ exist such that $\mu^{(k)}\in
ID$ and $\rh^{(k)}\in ID^0$.
\end{cor}

\medskip
The following theorem supplements Theorems \ref{t4a} and \ref{t4b}.

\begin{thm}\label{t5c}
Assume $0<r<p$.  Then $\mu^{(k)}\in ID^0$ for all $k\in\Z$
if and only if either
{\rm(a)} $r^2>p^2 q$ or {\rm(b)} $p^2 q^2 <r^2\leq p^2 q$
 and $c^m\neq2$ for all $m\in\N$.
\end{thm}

\begin{proof}
We have $\mu^{(k)}\in ID\cup ID^0$ by Theorem~\ref{t4a}. It follows
from Theorem \ref{t5b} that $\mu^{(k)}\in ID^0$ for all $k\in\N$ if
and only if either (a) or (b) holds. If $\mu^{(k)}\in ID^0$ for all
$k\in\N$, then $\mu^{(k)}\in ID^0$ for all $k\in\Z$ by Theorem
\ref{t4b}.
\end{proof}

The limit distribution of $\mu^{(k)}$ as $k\to\infty$ is as follows.

\begin{thm}\label{t5d} Let $c,p,q,r$ be fixed.

{\rm(i)} Assume $q>0$. Define $(c^{\sharp}, p^{\sharp}, q^{\sharp}, r^{\sharp})
=(c,1-q,q,0)$ and let $\mu^{\sharp(k)}$ be the distribution corresponding to
$\mu^{(k)}$ with $(c^{\sharp}, p^{\sharp}, q^{\sharp}, r^{\sharp})$ used
in place of $(c,p,q,r)$.  Then $\mu^{(k)}$ weakly converges to $\mu^{\sharp(0)}$
as $k\to\infty$.

{\rm(ii)}  Assume $q=0$.  Then $\mu^{(k)}$ weakly converges to $\dl_0$
as $k\to\infty$.
\end{thm}

We remark that $\mu^{\sharp(k)}$ does not depend on $k$ and is
infinitely divisible, so that the limit distribution is infinitely
divisible in all cases, although by Theorems~\ref{t4a} and~\ref{t5c}
there are many cases of parameters for which $\mu^{(k)}\not\in ID$
for all $k\in \mathbb{Z}$.

\begin{proof}
It follows from \eqref{p2c6} that
\[
\wh\mu^{(k)}(z)=\wh\mu^{(0)}(c^{-k}z)\prod_{n=0}^{k-1}\frac{1-q}{1-qe^{ic^{-n}z}}
\]
for $k\in\N$. Hence, as $k\to\infty$ we obtain
\[
\wh\mu^{(k)}(z)\to \prod_{n=0}^{\infty}\frac{1-q}{1-qe^{ic^{-n}z}}=
\begin{cases} \wh\mu^{\sharp(0)}(z), &  q>0, \\
1 = \wh\delta_0 (z), & q= 0.\end{cases}
\]
\end{proof}

\noindent {\it Remark.} Let  the parameters $c,p,q,r$ be fixed and
$\{ Z_k , k \in \Z\}$ be a sequence of independent identically
distributed random variables,  geometrically distributed with
parameter $q$ if $q>0$ and distributed as $\delta_0$ if $q=0$. Let
$k_0 \in \Z$ and $X_{k_0}$ be a random variable with distribution
$\mu^{(k_0)}$, independent of $\{Z_k , k > k_0\}$. Define $\{ X_k, k
\geq k_0\}$ inductively by
\begin{equation*} \label{eq-time1}
X_{k+1} = c^{-1} X_k + Z_{k+1}, \quad k=k_0, k_0+1, \ldots.
\end{equation*}
Then $\law (X_k) = \mu^{(k)}$ for all $k\geq k_0$ by \eqref{p2c6},
so that the $\mu^{(k)}$ appear naturally as marginal distributions
of a certain autoregressive process of order 1. The limit
distributions $\mu^{\sharp(0)}$ ($q>0$) and $\delta_0$ ($q=0$) as
$k\to\infty$  described in Theorem~\ref{t5d}  give the unique
stationary distribution of the corresponding AR(1) equation
$$Y_{k+1} = c^{-1} Y_k + Z_{k+1} , \quad k\in \Z.$$

%%%%%%%%%%%%%%%%%%%%%%%%%%%%%%%%%%%%%%%%%%%
%%%%%%%%%%%%%%%%%%%%%%%%%%%%%%%%%%%%%%%%%%%Section 6
\section{Conditions for infinite divisibility for $k<0$}

In this section we obtain  necessary and sufficient conditions for
$\rh^{(k)}\in ID$ and $\mu^{(k)}\in ID$ when $k<0$ and then derive
some simple consequences of these characterizations. Again, by
virtue of Theorem~\ref{t4a}, we only have to consider the case
$0<r<p$.

\begin{thm}\label{t6a}
Let $k$ be a negative integer. Assume $0<r<p$. Then $\rh^{(k)}\in
ID$ if and only if $2c^{|k|}\in\N$ and
$$q^{2c^{|k|}}\geq c^{|k|} (r/p)^2.$$
\end{thm}

\begin{proof}
The proof of Theorem~\ref{t5a} shows that $\mu^{(k)} \in ID$ if and
only if $s:= 2c^{-k}\in \N$ and \eqref{eq-for-rho} holds. From $k<0$
we have  $s\geq3$ and hence
\[
q^s \geq (s/2)^{1/m} (r/p)^2, \quad \forall\; m\in\N
\quad\Leftrightarrow\quad q^s\geq (s/2)(r/p)^2,
\]
which completes the proof.
\end{proof}

The characterization when $\mu^{(k)} \in ID$ for negative $k$ is
much more involved and different techniques will be needed according
to whether $2c^j \in \N_{\rm even}$ or $2c^j \in \N_{\rm odd}$ for
some $j\in \N$. In the first case, the characterization
 will be achieved in terms
of the function $h_{\alpha,\gamma}$ defined below. Let $\al, \gamma
\in\N$ with $\al\geq2$.  We use the function
\begin{equation*}\label{6-1}
x\mapsto F_{\al}(x)=\sum_{n=0}^{\infty} \al^{-n}x^{2\al^n},\qquad
0\leq x\leq 1
\end{equation*}
and the functions $x\mapsto h_{\al,\gamma}(x)$ and $x\mapsto
f_{\alpha,\gamma}(x)$ for $0<x\leq1$ defined by the relations
\begin{eqnarray}\label{6-2}
\al^{-\gamma}F_{\al}(x)& = & F_{\al}(h_{\al,\gamma}(x)),\\
f_{\alpha, \gamma} (x) & = & x^{-1} h_{\alpha,\gamma}(x).\nonumber
\end{eqnarray}
Observe that $F_\alpha$ is strictly increasing and continuous on
$[0,1]$ with $F_\alpha(0) = 0$ and hence $h_{\al,\gamma}(x)$ is
uniquely definable for $x\in (0,1]$ and it holds
$0<h_{\al,\gamma}(x)<x$. The next proposition describes some
properties of $h_{\alpha,\gamma}$ which will be used in the sequel.

\begin{prop}\label{p6a}
The functions $h_{\al,\gamma}$ and $f_{\alpha, \gamma}$ are
continuous and strictly increasing on $(0,1]$ and satisfy
\begin{gather}
\lim_{x\downarrow 0}  f_{\al,\gamma}(x) = \al^{-\gamma/2},\label{p6a1}\\
f_{\alpha, \gamma} (1) = h_{\al,\gamma}(1)<\al^{-\gamma/4},\label{p6a2}\\
h_{\al,\gamma}(1)<\al^{-\gamma/2}(1+\al^{-1})\quad\text{for all
$\gm$ if $\al$ is large enough},
\label{p6a3}\\
h_{\alpha,\gamma} (x) > h_{\alpha,\gamma+1} (x), \qquad\forall\,
x\in (0,1],\label{eq-h-new}\\
 h_{\al,\gamma}(x^n)\geq (h_{\al,\gamma}(x))^n
,\qquad\forall\, x\in (0,1]\;\;\forall\,n\in\N.\label{p6a4}
\end{gather}
\end{prop}

\begin{proof}
Since $F_{\al}$ is a continuous strictly increasing function defined
on $[0,1]$, it follows that $h_{\al,\gamma}$ is continuous and
strictly increasing on $(0,1]$, and hence that $f_{\alpha, \gamma}$
is continuous. Also observe that $h_{\alpha,\gamma}(x) \to 0$ as
$x\downarrow 0$ since $F_\al(0) = 0$. From \eqref{6-2} and the
Taylor expansion of $F_\alpha$ we obtain as $x\dar0$,
\[
\al^{-\gamma} x^2 \big(1+O(x^{2(\al-1)})\big)=(h_{\al,\gamma}(x))^2
\big( 1+ O((h_{\al,\gamma}(x))^{2(\al-1)})\big)
\]
from which \eqref{p6a1} follows.  In order to show \eqref{p6a2},
first let us check that
\begin{equation}\label{p6a5}
F_{\al}(\al^{-\gamma/4})>\alpha^{1-\gamma}/(\al-1).
\end{equation}
Indeed, if $\al\geq3$ or $\gamma\geq 2$, then use $F_{\al}(x)>x^2$
and obtain
\[
F_{\al}(\al^{-\gamma/4})>\al^{-\gamma/2} = \alpha^{1-\gamma}
\alpha^{\gamma/2-1}>\alpha^{1-\gamma}/(\al-1),
\]
since $\al^{\gamma/2-1} \geq 1/(\al -1)$ if $\alpha\geq 3$ or
$\gamma\geq 2$. If $\al=2$ and $\gamma=1$, then use
$F_2(x)>x^2+x^4/2+x^8/4$ to obtain
\[
F_2(2^{-1/4})>2^{-1/2}+2^{-2}+2^{-4}=0.7071\cdots+0.25+0.0625=1.0196\cdots>1,
\]
which proves \eqref{p6a5}. Since $F_{\al}(1)=\al/(\al-1)$, we have
$F_{\al}(h_{\al,\gamma}(1))=\alpha^{1-\gamma}/(\al-1)$. Hence
\eqref{p6a2} follows from \eqref{p6a5}.

To see \eqref{p6a3} it is enough to show that
\begin{equation}\label{p6a6}
F_{\al}(\al^{-\gamma/2}(1+\al^{-1}))>\al^{1-\gamma} /
(\al-1)\quad\text{for all $\gm$ if $\al$ is large enough.}
\end{equation}
Since $F_{\al}(x)>x^2$, we have
\[
F_{\al}(\al^{-\gamma/2}(1+\al^{-1}))>\al^{-\gamma}(1+\al^{-1})^2>\al^{-\gamma}(1+2\al^{-1}).
\]
On the other hand,
\[
\alpha^{1-\gamma}/(\al-1)=\al^{-\gamma}(1+\al^{-1}+O(\al^{-2})),
\quad \al\to\infty.
\]
Hence \eqref{p6a6} holds.

To see \eqref{eq-h-new}, observe that
$$F_\alpha (h_{\alpha, \gamma} (x)) = \alpha^{-\gamma} F_\alpha(x) =
\alpha \, F_\alpha (h_{\alpha, \gamma+1}(x))$$ by \eqref{6-2}, which
together with the strict increase of $F_\alpha$ implies
\eqref{eq-h-new}.

Before we can prove \eqref{p6a4}, we need to show that $f_{\alpha,
\gamma}$ is strictly increasing. For that, let us first show that
\begin{equation}\label{p6a7}
f_{\al,\gamma}(x)<f_{\al,\gamma}(1),\quad\forall\, x\in(0,1).
\end{equation}
Suppose, on the contrary, that $f_{\al,\gamma}(x_0)\geq
f_{\al,\gamma}(1)$ for some $x_0\in(0,1)$.  Then
$\al^{-\gamma}F_{\al}(x_0)=F_{\al}(f_{\al,\gamma}(x_0)x_0) \geq
F_{\al}(f_{\al,\gamma}(1)x_0)$, that is,
\[
\sum_{n=0}^{\infty}
\al^{-n}(\al^{-\gamma}-f_{\al,\gamma}(1)^{2\al^n})
x_0^{2\al^n}\geq0.
\]
Let
\begin{equation*}\label{p6a8}
G_{\al,\gamma}(\xi)=\sum_{n=0}^{\infty}\al^{-n}(\al^{-\gamma}-f_{\al,\gamma}(1)^{2\al^n})
\xi^{2(\al^n-1)}, \quad \xi \in [0,1].
\end{equation*}
Then $G_{\al,\gamma} (x_0)\geq0$ and $G_{\al,\gamma}(1)=0$, which
follows from $\al^{-\gamma}F_{\al}(1)=F_{\al}(f_{\al,\gamma}(1))$.
But we have
\begin{equation}\label{p6a9}
G'_{\al,\gamma}(\xi)=\sum_{n=1}^{\infty}
2(\al^n-1)\al^{-n}(\al^{-\gamma}-f_{\al,\gamma}(1)^{2\al^n})
\xi^{2(\al^n-1)-1}>0,\quad \xi \in (0,1),
\end{equation}
since
\[
\al^{-\gamma}-f_{\al,\gamma}(1)^{2\al^n}\geq
\al^{-\gamma}-f_{\al,\gamma}(1)^{2\al}>\al^{-\gamma}-
\al^{-\al\gamma/2}\geq0
\]
for $n\geq1$ by \eqref{p6a2}. This is absurd. Hence \eqref{p6a7} is true.

Now we show that $f_{\al,\gamma}$ is strictly increasing on $(0,1]$.
Suppose that there exist $x_1$ and $x_2$ in $(0,1]$ such that
$x_1<x_2$ and $f_{\al,\gamma}(x_1)\geq f_{\al,\gamma}(x_2)$. Then
\begin{gather*}
\al^{-\gamma} F_{\al}(x_1)=F_{\al}(f_{\al,\gamma}(x_1) x_1)
\geq  F_{\al}(f_{\al,\gamma}(x_2) x_1),\\
\al^{-\gamma} F_{\al}(x_2)=F_{\al}(f_{\al,\gamma}(x_2) x_2),
\end{gather*}
that is,
\begin{gather*}
\al^{-\gamma} \sum_{n=0}^{\infty} \al^{-n} x_1^{2\al^{n}}\geq
\sum_{n=0}^{\infty} \al^{-n}f_{\al,\gamma}(x_2)^{2\al^{n}}
x_1^{2\al^{n}},\\
\al^{-\gamma} \sum_{n=0}^{\infty} \al^{-n} x_2^{2\al^{n}}=
\sum_{n=0}^{\infty}
\al^{-n}f_{\al,\gamma}(x_2)^{2\al^{n}}x_2^{2\al^{n}}.
\end{gather*}
Define
\begin{equation*}\label{p6a10}
H_{\al,\gamma}(\xi)=\sum_{n=0}^{\infty} \al^{-n}( \al^{-\gamma}
-f_{\al,\gamma}(x_2)^{2\al^{n}}) \xi^{2(\al^{n}-1)}, \quad \xi \in
[0,1].
\end{equation*}
Then we have $H_{\al,\gamma}(x_1)\geq0$ and $H_{\al,\gamma}(x_2)=
0$. On the other hand, noting that $f_{\al,\gamma}(x_2)\leq
f_{\al,\gamma}(1)$ by \eqref{p6a7}, we can prove
$H'_{\al,\gamma}(\xi)>0$ in the same way as the proof of
\eqref{p6a9}. This is
 a contradiction. Hence $f_{\al,\gamma}$ is strictly increasing.

Finally, \eqref{p6a4} is proved.  Indeed, this is trivial for $n=1$, and
for $n\geq2$ we have
\[
(h_{\al,\gamma}(x))^n = (f_{\alpha,\gamma}(x))^n x^n <
\al^{-n\gamma/4} x^n  \leq \al^{-\gamma/2} x^n <
f_{\al}(x^n) x^n = h_{\alpha, \gamma}(x^n),
\]
noting that $f_{\al,\gamma}(x)$ is strictly increasing and using
\eqref{p6a1} and \eqref{p6a2}.
\end{proof}

Now we give the classification when $\mu^{(k)}\in ID$ for $k<0$ and
$0<r<p$. As usual, for $x\in\R$ we shall denote by $\lfloor x
\rfloor$ the largest integer being smaller than or equal to $x$, and
by $\lceil x \rceil$ the smallest integer being greater than or
equal to $x$.

\begin{thm}\label{t6c}
Let $k$ be a negative integer.  Assume $0<r<p$.

{\rm (i)} If\/ $2c^j\not\in\N$ for all integers $j$ satisfying
$j\geq|k|$, then $\mu^{(k)}\in ID^0$.

{\rm (ii)} Suppose that $c^j\in\N$ for some $j\in\N$. Let $l$ be the
smallest of such $j$ and let $\al=c^{l}$, $\beta := \lceil
|k|/l\rceil$ and  $h_{\al,\beta}$ be defined by \eqref{6-2}. Then
$\mu^{(k)}\in ID$ if and only if $q>0$ and
\begin{equation} \label{char-ii} h_{\al,\beta
}(q^{\al^\beta})\geq r/p.
\end{equation}

 {\rm (iii)} Suppose
that $2 c^j \in \N_{\rm odd}$ for some $j\in \N$ with $j\geq |k|$.
Then %$j$ is the unique non-negative integer $n$ such that $2c^n \in
%\N_{\rm odd}$.
$c^{j'}\not\in\N$ for all $j'\in\N$, and $j\in\N$ satisfying
$2 c^j \in \N_{\rm odd}$ is unique.
Let $\alpha=c^j$. Then $2\alpha \in \N_{\rm odd}$ and
$2\alpha\geq 3$.

\noindent $\mbox{\rm (iii)}_1$ Suppose that $2\alpha \geq 5$. Then
$\mu^{(k)} \in ID$ if and only if
\begin{equation} \label{eq-char-3a}
 q^{2\alpha} +
(r/p)^{2\alpha} \geq {\alpha} \,(r/p)^2.
\end{equation}

\noindent $\mbox{\rm (iii)}_2$ Suppose that $2\alpha = 3$. For $m\in
\N$, denote by $t(m)$ the largest integer $t'$ such that $m$ is an
integer multiple of $2^{t'}$, and write $a_m := m^{-1} (q^m -
(-r/p)^m)$
 Then $\mu^{(k)} \in
ID$ if and only if
\begin{equation} \label{eq-char-3b}
\sum_{s=0}^{t(m)} a_{3^{s+1} 2^{-s} m} \geq (2m)^{-1} (r/p)^{2m} ,
\quad \forall\; m \in \{1,\ldots, 149\}.
\end{equation}
\end{thm}

\begin{proof}
For all cases (i) -- (iii) observe that we have \eqref{t4b3}
 and \eqref{t4b4} since $0<r<p$. Therefore
 \begin{equation}\label{t4c1}
\nu_{\mu^{(k)}} =\sum_{n=0}^\infty \sum_{m'=1}^\infty
a_{m'}\dl_{c^{-n}m'} +\sum_{s'=1}^{|k|} \sum_{m'=1}^\infty (m')^{-1}
(-1)^{m'+1} (r/p)^{m'} \dl_{c^{s'} m'}
\end{equation}
with
\begin{equation}\label{t5b2b}
a_{m'} = (m')^{-1} (q^{m'}  -( - r/p)^{m'}) .
\end{equation}
We have $\mu^{(k)} \in ID$ if and only if $\nu_{\mu^{(k)}}\geq 0$.

To prove (i), assume that $2c^j\not\in\N$ for $j\geq|k|$. Consider
$\nu_{\mu^{(k)}} (\{ 2c^{|k|}\})$. Let $E=\{s\in\{1,2,\cdots,|k|-1\}
\colon 2c^s\in\N_{\mrm{odd}}\}$. Since $2c^j\not\in\N$ for
$j\geq|k|$, \eqref{t4c1} gives
\[
\nu_{\mu^{(k)}}(\{2c^{|k|}\})\leq  -(r/p)^2 /2 +\sum_{s\in E} (2
c^s)^{-1} (r/p)^{2c^s}.
\]
But since $E$ contains at most one element, we have
$$\sum_{s\in E} (2c^s)^{-1} (r/p)^{2c^s} < (r/p)^2 /2,$$
so that $\nu_{\mu^{(k)}}(\{2c^{|k|}\}) < 0$. Hence $\mu^{(k)}\in
ID^0$.

Let us prove (ii). Assume that $c^j \in \N$ for some $j\in\N$ and
let $l,\alpha,\beta$ be as in the statement of the theorem.
 If $q=0$ or if $q>0$ and $r>pq$, then $\mu^{(0)}\not\in ID$ and
hence $\mu^{(k)}\not\in ID$ by Proposition \ref{p5a} and Theorem
\ref{t4b}. Since $h_{\alpha,\beta} (q^{\alpha^{\beta}}) <
q^{\alpha^{\beta}} < q$ for $q>0$, %by \eqref{p6a2} and the strict
%increase of $f_{\alpha,\beta}$,
condition \eqref{char-ii} implies
$r\leq pq$. Hence we may assume $q>0$ and $r\leq pq$ from now on,
which in particular implies $a_{m'}\geq 0$. Hence
\begin{equation}\label{eq-jm-0}
\mu^{(k)}\in ID\quad\Leftrightarrow\quad {\begin{split}
&\text{$\nu_{\mu{(k)}} ( \{z\}) \ge0$ for all $z$ of the form $z=2m c^s$}\\
&\text{with $s\in \{1,\ldots, |k|\}$, $m\in \N$.} \end{split} }
\end{equation}
For $s\in \{1,\ldots, |k|\}$ and $m\in \N$ denote
$$g(s,m) := \sum_{n\in \N_0: 2mc^{s+n} \in \N} a_{2m c^{s+n}} - \sum_{n\in
\{0,\ldots,s-1\}: 2m c^{n}\in \N} (2m)^{-1} c^{-n} (r/p)^{2mc^n}.$$
If %$z$ is of the form
$z=2mc^s$ with $s\in \{1,\ldots, |k|-1\}$ and $m\in\N$, but
cannot be written in the form $z=2{m'} c^{s'}$ with $m'\in \N$ and
$s' \in \{s+1,\ldots, |k|\}$, then $2mc^{s-s'}\not\in
\N_{\mrm{even}}$ for $s'\in \{s+1,\ldots, |k|\}$, and hence by
\eqref{t4c1}
\begin{equation} \label{eq-jm-1}
\begin{split}
&\nu_{\mu^{(k)}} (\{z\})\\
&\geq \sum_{n\in \N_0: 2mc^{s+n}\in\N} a_{2m
c^{s+n}} - \sum_{s'\in \{1,\ldots, s\}: 2m c^{s-s'}\in \N}
\frac{1}{2m c^{s-s'}} (r/p)^{2m c^{s-s'}} = g(s,m).
\end{split}
\end{equation}
If %$z$ is of the form
$z=2m c^{|k|}$ with $m\in \N$, then
\begin{equation}\label{eq-jm-2}
\begin{split}
&\nu_{\mu^{(k)}} (\{z\})\\
&= \sum_{n\in \N_0: 2mc^{|k|+n}\in\N} a_{2m
c^{|k|+n}} - \sum_{s'\in \{1,\ldots, |k|\}: 2m c^{|k|-s'}\in \N}
\frac{1}{2m c^{|k|-s'}} (r/p)^{2m c^{|k|-s'}}\\
&= g(|k|,m).
\end{split}
\end{equation}
%where we used the fact that $c^j\in\N$ if $j\in\N_0$ is an integer
%multiple of $l$, and that $c^j$ is irrational otherwise.
We claim
that for $s\in \{1,\ldots, |k|\}$ we have
\begin{equation} \label{eq-jm-3}
g(s,m) \ge 0, \quad \forall\; m\in \N \quad \Leftrightarrow \quad
h_{\alpha,\lceil s/l\rceil} (q^{\alpha^{\lceil s/l\rceil}}) \ge r/p.
\end{equation}
Once we have established \eqref{eq-jm-3}, then \eqref{eq-jm-0} and
\eqref{eq-jm-2} show that \eqref{char-ii} is necessary for
$\mu^{(k)} \in ID$, while \eqref{eq-jm-0} -- \eqref{eq-jm-2} show
that it is also sufficient, since monotonicity of $h_{\alpha, \lceil
s/l\rceil}$ and \eqref{eq-h-new} imply
$$h_{\alpha,\lceil s/l\rceil}
(q^{\alpha^{\lceil s/l\rceil}}) \ge h_{\alpha,\lceil s/l\rceil}
(q^{\alpha^{\beta}}) \ge h_{\alpha,\beta}
(q^{\alpha^{\beta}}),\qquad \forall\; s\in \{1,\ldots, |k|\}.$$
 To show
\eqref{eq-jm-3}, observe that for $j\in\N_0$ we have
$c^j=\alpha^{j/l}\in \N$ if and only if $j$ is an integer multiple
of $l$, and that $c^j$ is irrational otherwise. From this property,
we see that
$$g(s,m) = \sum_{n'=0}^\infty a_{2m \alpha^{\lceil s/l \rceil}
\alpha^{n'}} - \sum_{n'=0}^{\lfloor (s-1)/l\rfloor} (2m)^{-1}
\alpha^{-n'} (r/p)^{2m \alpha^{n'}}.$$ Observing that $\lfloor
(s-1)/l\rfloor = \lceil s/l\rceil -1$, we have for $s\in \{1,\ldots,
|k|\}$ and $m\in \N$,
\begin{align*}
\lefteqn{g(s,m) \ge 0}\\
 &\Leftrightarrow\quad  \sum_{n=0}^{\infty}
\frac{1}{2m \alpha^{\lceil s/l\rceil +n}} \left[ q^{2m\alpha^{\lceil
s/l\rceil +n}}-(r/p)^{2m\alpha^{\lceil s/l \rceil +n}}\right]
\ge  \sum_{n=0}^{\lceil s/l \rceil -1} \frac{1}{2m\alpha^n} (r/p)^{2m\alpha^{n}}\\
&\Leftrightarrow\quad \alpha^{-\lceil s/l \rceil}
\sum_{n=0}^{\infty} \alpha^{-n} (q^{m\alpha^{\lceil
s/l\rceil}})^{2\alpha^{n}}\ge \sum_{n=0}^{\infty} \alpha^{-n}
\left[(r/p)^m\right]^{2\alpha^{n}}\\
&\Leftrightarrow\quad \alpha^{-\lceil s/l\rceil} F_{\alpha}
(q^{m \alpha^{\lceil s/l \rceil}})\ge F_{\alpha}((r/p)^m)\\
&\Leftrightarrow\quad h_{\alpha,\lceil s/l\rceil}(q^{m\alpha^{\lceil
s/l \rceil}}) \ge (r/p)^m.
\end{align*}
Now \eqref{eq-jm-3} follows from property \eqref{p6a4} of
$h_{\alpha, \lceil s/l\rceil}$, completing the proof of (ii).

Let us prove (iii).  Assume that $2c^{j''} \in \N_{\rm odd}$ for
some $j''\in \N$ with $j''\geq |k|$.
% It is easy to see that $j$ must then be
% the unique positive integer $n$ such that $2c^n \in \N_{\rm odd}$,
% that $2c^n$ is irrational for $n\in \N$ unless $n$ is an integer
% multiple of $j$, and that $2c^n\in \mathbb{Q} \setminus \N$ for
% $n\in \{2j, 3j, 4j, \ldots\}$.
Let $j$ be the smallest positive integer such that $c^j\in\Q$. Then
$c^{j'}\in\Q$ with $j'\in\N$ if and only if $j'$ is an integer
multiple of $j$.  We have $c^j\not\in\N$, since $2c^{j''} \in
\N_{\rm odd}$ for some $j''\in \N$.  Denote $c^j=a'/b'$ with
$a',b'\in\N$ having no common divisor.  Then $2c^{nj}=
2(a'/b')^n\not\in\N$ for all $n\in\N$ with $n\geq2$. Hence
$c^{j'}\not\in\N$ for all $j'\in \N$ and $2c^{j'} \in \N_{\rm odd}$
if and only if $j'=j$, so that $j=j''$. As in the proof of (ii), if
$q=0$, or if $q>0$ and $r> pq$, then $\mu^{(k)} \not\in ID$. On the
other hand, \eqref{eq-char-3a} for $2\alpha\geq 5$ clearly implies
$r\leq pq$, and as will be shown later in Equation~\eqref{eq-26} in
the proof of $\mbox{\rm (iii)}_2$, \eqref{eq-char-3b} implies $q^3
\geq (r/p)^2 (5/4)^{1/75}$.
 Hence we may and do assume $q>0$
and $r\leq pq$ from now on (this assumption will not be needed when
\eqref{eq-26} is derived from \eqref{eq-char-3b} in $\mbox{\rm
(iii)}_2$). In particular, $a_{m'} \geq 0$ for $m'\in \N$, and
$\mu^{(k)}\in ID$ is characterized by the right-hand side of
\eqref{eq-jm-0}. But as follows from the discussion above and the
fact that $j\geq |k|$, if
% $z$ is of the form
$z=2mc^s$ with $s\in \{1,\ldots, |k|\}$ and $m\in \N$,
then $z=c^{s'} m'$ for $s'\in\{1,\ldots, |k|\}$ and $m'\in \N$ if
and only if $s'=s$ and $m'=2m$. Further, since $c^{s+n} \in
\mathbb{Q}$ for $s\in \{1,\ldots, |k|\}$ and $n\in \N_0$ if and only
if  $s+n = (s'+1) j$ for some $s'\in \N_0$, in which case $c^{s+n} =
\alpha^{s'+1}$, \eqref{t4c1} gives
\begin{equation} \label{7a}
\nu_{\mu^{(k)}} ( \{ 2mc^s\} ) = \sum_{s'\in \N_0 :
2m\alpha^{s'+1}\in \N} a_{2m\al^{s'+1}} - (2m)^{-1} (r/p)^{2m}
\end{equation}
for $s\in \{1,\ldots,|k|\}$ and $m\in\N$.
Observe that this quantity does not depend on $s$.
Denote by $t(m)$ the largest integer $t'$ such that $m$ is
an integer multiple of $2^{t'}$, and observe that $2m \alpha^{s'+1}
= (2\alpha)^{s'+1} 2^{-s'} m \in \N$ with $s'\in \N_0$ if and only if
$s'\leq t(m)$ due to the assumption $2\alpha \in \N_{\rm odd}$. From
\eqref{eq-jm-0} and \eqref{7a} we hence conclude that
\begin{equation}\label{eq-1}
\mu^{(k)}\in ID\quad\Leftrightarrow\quad \sum_{s=0}^{t(m)}
a_{2m\alpha^{s+1}}\geq (2m)^{-1} (r/p)^{2m}  \text{ for all $m\in
\N$}.
\end{equation}

$\mbox{\rm (iii)}_1$ Now assume that $b := 2\alpha \geq 5$. If
$\mu^{(k)}\in ID$, then \eqref{eq-1} with $m=1$ gives
\eqref{eq-char-3a}, so that \eqref{eq-char-3a} is necessary for
$\mu^{(k)}\in ID$. To show that it is also sufficient, assume that
\eqref{eq-char-3a} holds for the rest of the proof of $\mbox{\rm
(iii)}_1$. We first claim that we have
\begin{equation} \label{eq-11}
q^b \geq (r/p)^b + (r/p)^2.
\end{equation}
Indeed, if $b\geq 7$, then \eqref{eq-char-3a} gives
$$q^b > (\frac{b}{2} -1) (r/p)^2 \geq 2 (r/p)^2 \geq (r/p)^b + (r/p)^2,$$
which is \eqref{eq-11}. If $b=5$, consider the function $\varphi: \R
\to \R$, $x\mapsto (5/2) x^2 - x^5$. Then $\varphi'(x) = 5x
(1-x^3)$, so that $\varphi$ is increasing on $[0,1]$. But
$\varphi\left( (3/4)^{1/3}\right) \approx 1.4 > 1,$ so that $(5/2)
(r/p)^2 - (r/p)^5 > 1$ whenever $r/p\geq (3/4)^{1/3}$. But since
$q<1$, it follows that \eqref{eq-char-3a} for $b=5$ implies $r/p
\leq (3/4)^{1/3}$, and hence \eqref{eq-char-3a} gives
\begin{equation*}
q^5 - \left[ (r/p)^5 + (r/p)^2 \right]
 \geq  (3/2) (r/p)^2 - 2 (r/p)^5
=  (3/2) (r/p)^2 \left[ 1 - (4/3) (r/p)^3 \right] \geq 0,
\end{equation*}
which is \eqref{eq-11} also in the case $b=5$.

Denote $ d := b \log (qp/r)$, $ \delta  :=  (b -2) \log (p/r)$, and
define the functions $\varphi_1, \varphi_2: \R \to \R$ by
$\varphi_1(x) = e^{dx}+1$ and $\varphi_2(x) = (b/2) e^{\delta x}$,
respectively. %Then the graphs of $\varphi_1$ and $\varphi_2$ have at
%most two intersection points.
Observe that $\varphi_1(0) = 2 < (b/2)
= \varphi_2(0)$. Further, \eqref{eq-char-3a} gives $\varphi_1(1)
\geq \varphi_2(1)$, and \eqref{eq-11} shows that $d>\delta$.
% so that
% $\lim_{x\to \infty} (\varphi_1(x) - \varphi_2(x)) = +\infty$.
Let $x_0\in(0,1]$ be the point satisfying $\ph_1(x_0)=\ph_2(x_0)$ and
$\ph_1(x)<\ph_2(x)$ for $x\in [0,x_0)$.  Then
$\ph'_1(x_0)\geq \ph'_2(x_0)$, that is, $de^{(d-\dl)x_0}\geq
\dl b/2$, which implies $de^{(d-\dl)x}>\dl b/2$ for $x>x_0$.
Hence, for all  $x> x_0$, $\ph'_1(x)> \ph'_2(x)$ and $\ph_1(x)>\ph_2(x)$.
This gives
\begin{equation*}
\left( \frac{qp}{r} \right)^{b m} + 1 \geq \frac{b}{2} \left(
p/r\right)^{(b -2)m}, \quad \forall\; m\in \N_{\rm odd},
\end{equation*}
so that $a_{bm} \geq (2m)^{-1} (r/p)^{2m}$ for $m\in \N_{\rm odd}$
 which is the right-hand side of \eqref{eq-1} for $m\in
\N_{\rm odd}$. Hence it only remains to show that the right-hand
side of \eqref{eq-1} holds for all $m\in \N_{\rm even}$, too. Since
$\sum_{s=0}^{t(m)} a_{2m\alpha^{s+1}} \geq a_{2m \alpha}= a_{bm}$,
by induction it is enough to prove the following for $m\in\N$:
\begin{equation} \label{eq-indu}
\text{If}\; a_{bm} \geq (2m)^{-1} (r/p)^{2m},\text{ then}\; a_{2bm}
 \geq (4m)^{-1} (r/p)^{4m}.
\end{equation}
So assume that $a_{bm} \geq (2m)^{-1} (r/p)^{2m}$. If
$m\in \N_{\rm even}$, this means that $q^{b m} - (r/p)^{b m} \geq
(b/2) (r/p)^{2m}$, and it follows that
\begin{eqnarray*}
a_{ 2bm} %& = & \frac{1}{2b m} \left( q^{2b m} -
%(r/p)^{2b m}\right) \\
& = & \frac{1}{2b m} \left[ q^{bm} - (r/p)^{b m}
\right] \, \left[ q^{b m} + (r/p)^{b m} \right]\nonumber \\
&{\geq} & \frac{1}{4m} (r/p)^{2m} \left[
q^{b m} + (r/p)^{b m} \right] \nonumber \\
& \geq & \frac{1}{4m} (r/p)^{2m} (r/p)^{2m} = \frac{1}{4m}
(r/p)^{4m}\nonumber ,
\end{eqnarray*}
where the last inequality follows from \eqref{eq-11}. If $m\in
\N_{\rm odd}$, then $a_{bm} \geq (2m)^{-1} (r/p)^{2m}$ means that
$q^{b m} + (r/p)^{b m} \geq (b/2) (r/p)^{2m}$, so that
\begin{eqnarray*}
a_{2bm} & = & \frac{1}{2b m} \left[ q^{b m} + (r/p)^{b m}
\right] \, \left[ q^{b m} - (r/p)^{b m} \right] \\
& \geq & \frac{1}{4m} (r/p)^{2m} \left[ q^{b m} - (r/p)^{b m}
\right]\\
& \geq &  \frac{1}{4m} (r/p)^{4m},
\end{eqnarray*}
where the last inequality follows from the fact that \eqref{eq-11}
implies
\begin{equation*}
q^b \geq \left[ (r/p)^{b m} + (r/p)^{2m}
\right]^{1/m}.\end{equation*} This completes the proof of
\eqref{eq-indu}, and it follows that \eqref{eq-char-3a} implies the
right-hand side of \eqref{eq-1} for all $m\in \N$, so that
\eqref{eq-char-3a} is also sufficient for $\mu^{(k)} \in ID$.

$\mbox{\rm (iii)}_2$ Now assume that $b := 2\alpha = 3$. Clearly,
\eqref{eq-char-3b} is nothing else than the right-hand side of
\eqref{eq-1} for $m\leq 149$, showing that \eqref{eq-char-3b} is
necessary for $\mu^{(k)} \in ID$. For the converse, assume that
\eqref{eq-char-3b} holds (without assuming a priori that $r\leq
pq$). Then \eqref{eq-char-3b} applied with $m=2$ gives $a_6 + a_9
\geq 4^{-1} (r/p)^4$, which is equivalent to
$$q^6 + (2/3) q^9  \geq  (3/2) (r/p)^4 + (r/p)^6 - (2/3)
(r/p)^9.$$ But since $q<1$ and $r/p < 1$, this implies
\begin{equation} \label{eq-41}
5/3 \geq  (r/p)^4 \left[ 3/2 + (r/p)^2 - (2/3) (r/p)^5 \right] \geq
(r/p)^4 \left[ 3/2 + (r/p)^2 - 2/3 \right].
\end{equation}
Using $x^4 \left[ 5/6\, + x^2 \right] \geq 1.6686... > 5/3$ for
$x\geq (13/14)^{1/4}$, \eqref{eq-41} gives
\begin{equation}
r/p  <  (13/14)^{1/4}. \label{eq-25}
\end{equation}
Applying \eqref{eq-char-3b} with $m=75$, i.e. using $a_{225} \geq
150^{-1} (r/p)^{150}$, gives
\begin{equation} \label{eq-42}
q^3  \geq (r/p)^2 \left[ 3/2 - (r/p)^{75} \right]^{1/75}.
\end{equation}
An application of \eqref{eq-25} shows that
$${\left[ 3/2 - (r/p)^{75} \right]^{1/75}}
 \geq  \left[ 3/2 - (13/14)^{75/4} \right]^{1/75}
 =  \left[ 3/2 - 0.2491...\right]^{1/75}  \geq (5/4)^{1/75},$$
which together with \eqref{eq-42} results in \begin{equation} q^3
\geq (r/p)^2 (5/4)^{1/75}. \label{eq-26}
\end{equation}
Now if $m\geq 150$, it follows from \eqref{eq-25} that
\begin{eqnarray*}
\left[ 3/2 + (r/p)^m \right]^{1/m} & {\leq} &
\left[ 3/2 + \left( 13/14\right)^{m/4}\right]^{1/m} \\
& \leq & \left[ 3/2 + \left( {13/14} \right)^{150/4} \right]^{1/150}
=  (1.2498...)^{1/75} < (5/4)^{1/75}.
\end{eqnarray*}
Together with \eqref{eq-26} this shows that
$$q^3  \geq  (r/p)^2 \left[ 3/2 + (r/p)^m \right]^{1/m}, \quad
m\geq 150.$$ But for $m$ even, $m\geq 150$, the last equation is
equivalent to $a_{3m} \geq (2m)^{-1} (r/p)^{2m}$. On the other hand,
if $m$ is odd and $m\geq 150$, then \eqref{eq-26} gives
$$a_{3m} = \frac{1}{3m} \left[ q^{3m} + (r/p)^{3m} \right] \geq \frac{1}{3m} q^{3m}
\geq \frac{1}{3m} (r/p)^{2m}   (5/4)^{m/75} \geq \frac{1}{2m}
(r/p)^{2m},$$ where we used $(5/4)^2 \geq 3/2$ in the last
inequality. Hence we obtain for $m\in \N$, $m\geq 150$, that
$$\sum_{s=0}^{t(m)} a_{3^{s+1} 2^{-s} m} \geq  a_{3m} \geq (2m)^{-1} (r/p)^{2m},$$
so that \eqref{eq-char-3b} implies the right-hand side of
\eqref{eq-1}. Hence \eqref{eq-char-3b} is sufficient for $\mu^{(k)}
\in ID$, completing the proof.
\end{proof}

\noindent{\it Remark}. The condition \eqref{eq-char-3a} in
Theorem~\ref{t6c} means $a_{2\alpha} \geq 2^{-1} (r/p)^2$ for $a_m$
of \eqref{t5b2b}, which together with $j\geq |k|$ completely
characterizes when $\mu^{(k)}\in ID$ in the case of $\mbox{(iii)}_1$
when $2\alpha\geq 5$. This is different in the case $2\alpha = 3$ of
$\mbox{(iii)}_2$. Here, the condition $a_{2\alpha} \geq 2^{-1}
(r/p)^2$  is not enough to ensure that $\mu^{(k)} \in ID$. For
example, if $q^3 > 1/2$ and $r/p= (13/14)^{1/4}$, then $a_{2\alpha}
\geq 2^{-1} (r/p)^2$, which is \eqref{eq-char-3b} for $m=1$, but
$\mu^{(k)} \not\in ID$, since \eqref{eq-char-3b} for $m=2$ implies
\eqref{eq-25} as shown in the proof of $\mbox{(iii)}_2$.
Nevertheless, there seems room to reduce the 149 conditions of
\eqref{eq-char-3b} to a smaller number, but we shall not investigate
this subject further.
\medskip

The following corollary gives handy sufficient and handy necessary
conditions for $\mu^{(k)} \in ID$.

\begin{cor} \label{cor-suf-nec}
Let $k$ be a negative integer and assume that $0<r<p$.

(i) Suppose that $c^j \in \N$ for some $j\in \N$. Let $l$ be the
smallest of such $j$ and let $\alpha = c^l$ and $\beta := \lceil
|k|/l \rceil$. Then $q^{\alpha^\beta} > \alpha^{\beta/4} (r/p)$ is a
necessary condition for $\mu^{(k)} \in ID$, while $q^{\alpha^\beta}
\geq \alpha^{\beta/2} (r/p)$ is a sufficient condition for
$\mu^{(k)} \in ID$.

(ii) Suppose that $2c^j \in \N_{\rm odd}$ for some $j\in\N$ with
$j\geq |k|$, and let $\alpha = c^j$. Then $q^\alpha > r/p$ is a
necessary condition for $\mu^{(k)} \in ID$, and $q^\alpha \geq
\alpha^{1/2} (r/p)$ is a sufficient condition for $\mu^{(k)} \in
ID$. If $2\alpha \geq 5$, then $q^\alpha > (\alpha-1)^{1/2} (r/p)$
is another necessary condition for $\mu^{(k)} \in ID$.
\end{cor}

\begin{proof}
To prove (i), observe that $f_{\alpha,\beta}$ is strictly increasing
by Proposition~\ref{p6a}, so that
$$\alpha^{-\beta/2} < q^{-\alpha^\beta} h_{\alpha,\beta}
(q^{\alpha^\beta}) < \alpha^{-\beta/4}$$ for $q\in (0,1)$ by
\eqref{p6a1} and \eqref{p6a2}. The assertion now follows from
\eqref{char-ii}.

To prove (ii), observe that by \eqref{eq-1} a necessary condition
for $\mu^{(k)} \in ID$ is that $a_{2\alpha m} \geq (2m)^{-1}
(r/p)^{2m}$ for all $m\in \N_{\rm odd}$. The latter condition is
equivalent to
$$\left[ q^{2\alpha} (p/r)^2 \right]^m + (r/p)^{(2\alpha -2)m} \geq
\alpha, \qquad\; \forall\; m\in \N_{\rm odd},$$ which shows that
$q^{2\alpha}
> (r/p)^2$ is a necessary condition for $\mu^{(k)}\in ID$ by
letting $m$ tend to infinity.  It is immediate from
\eqref{eq-char-3a}
that $q^\alpha > (\alpha-1)^{1/2} (r/p)$ is necessary for
$\mu^{(k)} \in ID$ if $2\alpha\geq 5$. If $2\alpha\geq 3$, then $j\geq |k|$ and
$q^\alpha \geq \alpha^{1/2} (r/p)$ imply
$$q^{2c^{|k|}} \geq q^{2 c^{j}} = q^{2\alpha} \geq \alpha
(r/p)^2 \geq c^{|k|} (r/p)^2,$$ so that $q^\alpha \geq \alpha^{1/2}(r/p)$
is a sufficient condition for $\rho^{(k)} \in ID$ by
Theorem~\ref{t6a} and hence for $\mu^{(k)} \in ID$.
\end{proof}

\noindent {\it Remark.} In the case of
Corollary~\ref{cor-suf-nec}~(i), another necessary condition for
$\mu^{(k)} \in ID$ is that $q^{\alpha^\beta} > \alpha^{\beta/2} (1+
\alpha^{-1})^{-1} (r/p)$, provided $\alpha$ is large enough. The
proof is the same but using \eqref{p6a3} instead of \eqref{p6a2}.
Compare with the sufficient condition $q^{\alpha^\beta}\geq
\alpha^{\beta/2} (r/p)$.\medskip

The following corollary is immediate from Theorems~\ref{t4a},
\ref{t6a} and \ref{t6c}.

\begin{cor}\label{c6a}
If $k$ is a negative integer, then parameters $c,p,q,r$ exist such
that $\mu^{(k)} \in ID$ and $\rho^{(k)} \in ID^0$.
\end{cor}

The following Theorem complements Theorems~\ref{t4b} and ~\ref{t5c}.

 \begin{thm}\label{t4c} Let $c>1$ and $p,q,r$ be fixed such that $p,r>0$ and
$p\neq r$. Then there is $k_0\in\Z$ such that $\mu^{(k)}\in ID^0$
for all $k\in\Z$ with $k<k_0$.
\end{thm}

\begin{proof}
By Theorem~\ref{t4a} it only remains to consider the case $0<r<p$.
Since a sequence $\{j_k, k\in \N\}$ of integers tending to $\infty$
such that $2c^{j_k} \in \N$ for all $k$ can only exist if $c^j\in
\N$ for some $j\in \N$, Theorem~\ref{t6c}~(i) gives the assertion
unless $c^j\in \N$ for some $j\in \N$. In the latter case, let
$\alpha$ and $l$ be defined defined as in Theorem~\ref{t6c}~(ii) and
$\beta_k := \lceil |k|/l\rceil$. Then $\beta_k\to \infty$ as
$k\to-\infty$ and hence $h_{\alpha,\beta_k}(q^{\alpha^{\beta_k}})\to
0$ as $k\to-\infty$ by \eqref{6-2}. In particular, \eqref{char-ii}
is violated for large enough $|k|$.
\end{proof}

%%%%%%%%%%%%%%%%%%%%%%%%%%%%%%%%%%%%%%
%%%%%%%%%%%%%%%%%%%%%%%%%%%%%%%%%%%%%%Section 7
\section{Symmetrizations}

In general, the symmetrization
$\sg^{\mrm{sym}}$ of a distribution $\sg$ is defined to be the distribution with
characteristic function $|\wh\sg(z)|^2$.
It is clear that
\begin{equation}\label{7-0}
\text{if $\sg\in ID$, then $\sg^{\mrm{sym}}\in ID$.}
\end{equation}
It follows from \eqref{1-8} that
\begin{equation}\label{7-1}
\wh\mu^{(k)\,\mrm{sym}}(z)=\wh\rh^{(k)\,\mrm{sym}}(z)\,\wh\mu^{(k)\,\mrm{sym}}
(c^{-1}z)
\end{equation}
for all $k\in\Z$, where $\wh\rh^{(k)\,\mrm{sym}}(z)$ and
$\wh\mu^{(k)\,\mrm{sym}}(z)$ denote the
characteristic functions of $\rh^{(k)\,\mrm{sym}}$ and $\mu^{(k)\,\mrm{sym}}$.
Thus $\mu^{(k)\,\mrm{sym}}$ is again $c^{-1}$-decomposable.
These symmetrizations have the
following remarkable property.

\begin{lem}\label{l7a}
Define $(c',p',q',r')=(c,r,q,p)$ and let $\rh^{\prime(k)}$ and $\mu^{\prime(k)}$
be the distributions corresponding to $\rh^{(k)}$ and $\mu^{(k)}$ with
$(c',p',q',r')$ used in place of $(c,p,q,r)$. Let $\rh^{\prime(k)\,\mrm{sym}}$
and $\mu^{\prime(k)\,\mrm{sym}}$ be their symmetrizations.
Then
\begin{align}
\rh^{\prime(k)\,\mrm{sym}}&=\rh^{(k)\,\mrm{sym}},\label{l7a1}\\
\mu^{\prime(k)\,\mrm{sym}}&=\mu^{(k)\,\mrm{sym}}\label{l7a2}
\end{align}
for $k\in\Z$.
\end{lem}

\begin{proof}
It follows from \eqref{p2c2} that
\[
\wh\rh^{(k)\,\mrm{sym}}(z)=\left| \frac{p+re^{ic^{-k} z}}{1-qe^{iz}} \right|^2
=\left| \frac{pe^{-ic^{-k} z}+r}{1-qe^{iz}} \right|^2
=\left| \frac{r+pe^{ic^{-k} z}}{1-qe^{iz}} \right|^2 .
\]
Hence $\rh^{\prime(k)\,\mrm{sym}}$ and $\rh^{(k)\,\mrm{sym}}$ have an
identical characteristic function, that is, \eqref{l7a1} is true.
Then \eqref{l7a2} follows as in \eqref{p2b2}.
\end{proof}

We also use the following general result.

\begin{lem}\label{l7b}
Suppose that $\sg$ is a distribution on $\R$.

{\rm (i)} If $\sigma \in ID \cup ID^0$, then $\sigma^{\rm sym} \in
ID\cup ID^0$.

{\rm (ii)} If $\sigma\in ID^0$
 with quasi-L\'evy measure being concentrated on $(0,\infty)$,
then $\sg^{\rm{sym}} \in ID^0$.
\end{lem}

\begin{proof}
(i) It is clear that if $\sigma$ satisfies \eqref{d1a1} with
$\gamma$, $a$ and $\nu_\sigma$, then $\sigma^{\rm{sym}}\in ID \cup
ID^0$ satisfying \eqref{d1a1} with $\gamma^{\rm sym} = 0, a^{\rm
sym} = 2a$ and $\nu_{\sg^{\mrm{sym}}}$ given by $
\nu_{\sg^{\mrm{sym}}}(B)=\nu_{\sg}(B)+\nu_{\sg}(-B)$ for  $B\in\mcal
B(\R)$.

(ii) If $\sigma \in ID^0$, then $a<0$ or $\nu_\sigma$ has nontrivial
negative part. Hence it follows from the proof of (i) that if $a<0$,
then $a^{\rm sym} <0$, and if $\nu_\sigma$ has nontrivial negative
part and is concentrated on $(0,\infty)$, then $\sg^{\mrm{sym}}$ has
non-trivial negative part. In both cases it holds $\sigma^{\rm sym}
\in ID^0$.
\end{proof}

\begin{thm}\label{t7A} Let $k\in\Z$.

{\rm(i)} If $p=0$ or if $r=0$, then $\rh^{(k)\,\mrm{sym}}$ and
$\mu^{(k)\,\mrm{sym}}$ are in $ID$.

{\rm(ii)} If $p\neq r$, then  $\rh^{(k)\,\mrm{sym}}$ and $\mu^{(k)\,\mrm{sym}}$
 are in $ID \cup ID^0$.

{\rm(iii)} If $p=r$, then  $\rh^{(k)\,\mrm{sym}}$ and
 $\mu^{(k)\,\mrm{sym}}$ are in $ID^{00}$.
\end{thm}

\begin{proof}
(i) and (ii) are clear from Theorem \ref{t4a} (i)-(iii) and
Lemma~\ref{l7b}, while (iii) follows from the fact that
$\wh\rh^{(k)}(z)$ and hence $\wh\mu^{(k)}(z)$ have zero points for
\(p=r\) by \eqref{p2c2}.
\end{proof}

In studying infinite divisibility properties of $\rh^{(k)\,\mrm{sym}}$ and
$\mu^{(k)\,\mrm{sym}}$, we will only consider whether they are infinitely divisible
or not in the case where
\begin{equation}\label{7-2}
\text{$p>0$,\quad $r>0$,\quad and\quad $p\neq r$},
\end{equation}
as we have Theorem \ref{t7A}.

\begin{thm}\label{t7K} Let $k\in \mathbb{Z}$ and assume
\eqref{7-2}. Let $\rh^{\prime(k)}$ and $\mu^{\prime(k)}$ be defined
as in Lemma~\ref{l7a}.

{\rm (i)} $\rho^{(k)\,\mrm{sym}}\in ID$ if and only if
$\rho^{(k)}\in ID$ or  $\rh^{\prime(k)}\in ID$.

{\rm (ii)} $\mu^{(k)\,\mrm{sym}}\in ID$ if only if $\mu^{(k)}\in ID$
or $\mu^{\prime(k)}\in ID$.
\end{thm}

\begin{proof}
% If $r<p$, then $\rho^{(k)}$ and $\mu^{(k)}$ are in $ID \cup ID^0$
% with quasi-L\'evy measures being concentrated on $(0,\infty)$ by
% Theorem~\ref{t4a}~(ii), and from Lemma~\ref{l7b} it follows that
% $\rho^{(k)\,\mrm{sym}}\in ID$ if and only if $\rho^{(k)}\in ID$, and
% $\mu^{(k)\, \mrm{sym}} \in ID$ if and only if $\mu^{(k)} \in ID$.
% Observe further that $\rho^{\prime(k)}, \mu^{\prime(k)} \in ID^0$
% for $r<p$ by Theorem~\ref{t4a}~(iii), since $r<p$ implies $r'>p'$.
% If $r>p$, then $r' < p'$, and the same reasoning can be applied to
% $\rh^{\prime(k)\,\mrm{sym}}$ and $\mu^{\prime(k)\,\mrm{sym}}$, so
% that the result follows from Lemma~\ref{l7a}.
The \lq if' part of (i) follows from \eqref{7-0} and \eqref{l7a1}.
To see the \lq only if' part, suppose that $\rh^{(k)\,\mrm{sym}}
\in ID$.  If $r<p$, then $\rh^{(k)}\in ID\cup ID^0$
with quasi-L\'evy measure being concentrated on $(0,\infty)$ by
Theorem~\ref{t4a}~(ii), and $\rh^{(k)}\in ID$ from Lemma~\ref{l7b}~(ii).
If $r>p$, then $r'<p'$ and the same reasoning for $\rh^{\prime (k)}$
combined with \eqref{l7a1} shows that $\rh^{\prime (k)}\in ID$.
Hence (i) is true.  We obtain (ii) in the same way.
\end{proof}

We can now give necessary and sufficient conditions for
$\rho^{(k)\,\mrm{sym}}$ and $\mu^{(k)\,\mrm{sym}}$ being infinitely
divisible.  For $k=0$ in (i) below, the corresponding conditions
were already obtained in Theorem~2.2 of \cite{LS}, but thanks to
Theorem~\ref{t7K}, a new and much shorter proof can now be given for
that part of Theorem~2.2 in \cite{LS}.

\begin{thm}\label{t7B} Let $k\in \mathbb{Z}$ and assume \eqref{7-2}.

{\rm(i)} Let $k=0$. If $(r/p)\land (p/r)\leq q$, then
$\rh^{(0)\,\mrm{sym}}, \mu^{(0)\,\mrm{sym}}\in ID$. Conversely, if
$(r/p)\land (p/r)> q$, then $\rh^{(0)\,\mrm{sym}},
\mu^{(0)\,\mrm{sym}}\in ID^0$.

{\rm (ii)} Let $k>0$. Then $\rh^{(k)\,\mrm{sym}}\in ID$ if and
only if $c^k=2$ and $(r/p)^2 \land (p/r)^2\leq q$.

{\rm (iii)} Let $k>0$. Then $\mu^{(k)\,\mrm{sym}}\in ID$ if and
only if one of the following holds: {\rm(a)}~$(r/p)\land(p/r)\leq
q$; {\rm(b)} $c^l=2$ for some $l\in\{ 1,2,\ldots,k\}$
 and $(r/p)^2 \land (p/r)^2\leq q$.

 {\rm (iv)} Let $k<0$.
Then $\rh^{(k)\,\mrm{sym}}\in ID$ if and only if $2c^{|k|}\in\N$ and
$q^l\geq (l/2)[(r/p)^2\land (p/r)^2]$ for $l=2c^{|k|}$.

{\rm (v)} Let $k<0$. If $2c^j \not\in \N$ for
all integers $j$ satisfying $j\geq |k|$, then $\mu^{(k)\,\mrm{sym}}
\in ID^0$.

 {\rm (vi)} Let $k<0$. Suppose that
$c^j\in\N$ for some $j\in\N$. Let $l$ be the smallest of such $j$
and let $\al=c^{l}$, $\beta := \lceil |k|/l\rceil$ and
$h_{\al,\beta}$ be defined by \eqref{6-2}. Then
$\mu^{(k)\,\mrm{sym}}\in ID$  if and only if $q>0$ and
$h_{\al,\beta}(q^{\al^\beta})\geq (r/p)\land (p/r)$.

{\rm (vii)} Let $k<0$. Suppose that
$2c^j\in\N_{\rm odd}$ for some $j\in\N$ with $j\geq |k|$. Then $j$
is unique. Let $\alpha = c^j$ and suppose that $2\alpha \geq 5$.
Then $\mu^{(k)\,\mrm{sym}}\in ID$  if and only if %one of the
% following holds: {\rm(a)} $r<p$ and $q^{2\alpha} + (r/p)^{2\alpha}
% \geq \alpha (r/p)^2$; {\rm (b)}~$r>p$ and $q^{2\alpha} +
% (p/r)^{2\alpha} \geq \alpha (p/r)^2$.
$q^{2\alpha} + ((r/p)\land (p/r))^{2\alpha}
\geq \alpha ((r/p)\land (p/r))^2$.
\end{thm}

\begin{proof}
All assertions are immediate consequences of Theorem~\ref{t7K},
Theorem~\ref{t4a}, and the corresponding results obtained earlier.
For (i), use Proposition~\ref{p5a}, for (ii) Theorem~\ref{t5a}, for
(iii) Theorem~\ref{t5b}, for (iv) Theorem~\ref{t6a}, and for (v) --
(vii) use
 Theorem~\ref{t6c}.
\end{proof}

Conditions for $\mu^{(k)\,\mrm{sym}}\in ID$ when $2c^j=3$ with
$j,-k\in \N$ and $j\geq |k|$ can be written down similarly as in
(vii) above with the aid of Theorem~\ref{t6c}~$\mbox{(iii)}_2$.

\begin{cor}\label{c7a}
For each $k\in\Z\setminus \{0\}$, parameters $c,p,q,r$ exist such
that $\mu^{(k)\,\mrm{sym}}\in ID$ and $\rho^{(k)\,\mrm{sym}}\in
ID^0$.
\end{cor}

The proof is immediate from Theorem~\ref{t7B}.  %To the best of our
% knowledge, Corollary~\ref{c7a} gives the first example of a
% symmetric, infinitely divisible distribution $b$-decomposable
% distribution without infinitely divisible factor, showing that the
% phenomenon first observed by Niedbalska-Rajba~\cite{Ni} may also
% occur for symmetric distributions.
Corollary~\ref{c7a} gives symmetric examples of  infinitely divisible
distributions which are $b$-decomposable without infinitely divisible
factor, the phenomenon first observed by Niedbalska-Rajba~\cite{Ni}.

The next corollary gives further examples of a phenomenon first
observed by Gnedenko and Kolmogorov~\cite{GK}, p.~82. Its proof is
immediate from Theorem~\ref{t4a}, Proposition \ref{p5a}, Theorems
\ref{t5a}, \ref{t5b}, and \ref{t7B}.

\begin{cor}
For each $k\in\Z$, there is a case where $\rh^{(k)\,\mrm{sym}}\in
ID$ with $\rh^{(k)}\in ID^{0}$ and there is a case where
$\mu^{(k)\,\mrm{sym}}\in ID$ with $\mu^{(k)}\in ID^{0}$.
\end{cor}

Let us give the analogues of Theorems \ref{t4b}, \ref{t4c}, and \ref{t5c}.

\begin{thm}\label{t7G}
Let $k\in\Z$ and the parameters $c, p,q,r$ be fixed.
 If $\mu^{(k)\,\mrm{sym}}\in ID$, then $\mu^{(k+1)\,\mrm{sym}}\in ID$.
\end{thm}

\begin{proof}
From \eqref{p2c6} follows
\[
\wh\mu^{(k+1)\,\mrm{sym}}(z)=\wh\mu^{(k)\,\mrm{sym}}(c^{-1}z)
\left|\frac{1-q}{1-qe^{iz}}\right|^2,
\]
and the second factor in the right-hand side is an infinitely divisible
characteristic function.
\end{proof}

\begin{thm}\label{t7H}
Let $c>1$ and the parameters $p,q,r$ be fixed such that $p>0$ and
$r>0$. Then there is $k_0\in\Z$ such that, for every $k\in\Z$ with
$k<k_0$, $\mu^{(k)\,\mrm{sym}}\not\in ID$.
\end{thm}

\begin{proof} For $r=p$, the assertion is obvious  by
Theorem~\ref{t7A}. For $r\neq p$ it follows from Theorems~\ref{t4a},
\ref{t4c} and \ref{t7K}.
\end{proof}

\begin{thm}\label{t7I}
Assume \eqref{7-2}.  Then $\mu^{(k)\,\mrm{sym}}\in ID^0$  for all
$k\in\Z$ if and only if one of the following holds: {\rm(a)}
$(r/p)^2\land (p/r)^2 > q$; {\rm(b)} $(r/p)\land (p/r) > q$ and
$c^m\neq2$ for all $m\in\N$.
\end{thm}

\begin{proof}
For fixed $k\in\N$, it follows from Theorem \ref{t7B}~(iii) that
$\mu^{(k)\,\mrm{sym}}$ is non-infinitely divisible if and only if
one of the following holds: {\rm(a)} $(r/p)^2\land(p/r)^2> q$;
{\rm(b)} $(r/p)\land (p/r) > q$ and $c^m\neq2$ for all
$m\in\{1,2,\ldots,k\}$. Our assertion is obtained from this.
\end{proof}

Some continuity properties of the symmetrizations of $\mu^{(k)}$ are added.

\begin{thm}\label{t7d}
Let $k\in\Z$ and the parameters $c,p,q,r$ be fixed. Then:

{\rm(i)}  $\mu^{(k)\,\mrm{sym}}$ is absolutely continuous if and only if
$\mu^{(0)\,\mrm{sym}}$ is absolutely continuous.

{\rm(ii)}  $\mu^{(k)\,\mrm{sym}}$ is continuous-singular if and only if
$\mu^{(0)\,\mrm{sym}}$ is continuous-singular.

{\rm(iii)}\quad $\dim\,(\mu^{(k)\,\mrm{sym}})=\dim\,(\mu^{(0)\,\mrm{sym}})$.

{\rm(iv)}\quad $\dim\,(\mu^{(k)\,\mrm{sym}})\leq H(\rh^{(k)\,\mrm{sym}})/\log c
\leq 2H(\rh^{(k)})/\log c$.
\end{thm}

\begin{proof}
It follows from \eqref{p2c4} that
\[
\wh\mu^{(k)\,\mrm{sym}}(z)=\wh\mu^{(k+1)\,\mrm{sym}}(z)\,|p_0+r_0e^{ic^{-k}z}|^2,
\]
where $p_0=p/(p+r)$ and $r_0=r/(p+r)$. Since
\[
|p_0+r_0e^{ic^{-k}z}|^2=(p_0+r_0e^{ic^{-k}z})\,(p_0+r_0e^{-ic^{-k}z})
=p_0^2+r_0^2+p_0r_0(e^{ic^{-k}z}+e^{-ic^{-k}z}),
\]
we have
\[
\mu^{(k)\,\mrm{sym}}=\mu^{(k+1)\,\mrm{sym}} *[(p_0^2+r_0^2)\dl_0+
p_0r_0 (\dl_{c^{-k}}+\dl_{-c^{-k}})],
\]
that is,
\[
\mu^{(k)\,\mrm{sym}}(B)=(p_0^2+r_0^2)\mu^{(k+1)\,\mrm{sym}}(B) +
p_0r_0 [ \mu^{(k+1)\,\mrm{sym}}(B-c^{-k})+ \mu^{(k+1)\,\mrm{sym}}(B+c^{-k})]
\]
for $B\in\mcal B(\R)$. Hence an argument similar to the proof of
Theorem \ref{t3a} works to show (i)--(iii), since
$\mu^{(k)\,\mrm{sym}}$ is $c^{-1}$-decomposable by (\ref{7-1}) and
hence either absolutely continuous, continuous singular, or a Dirac
measure. Assertion (iv) follows from Watanabe's theorem \cite{Wa00}
and E29.23 of \cite{Sa}.
\end{proof}

% All statements in (i)--(vi) of Remark below Theorem \ref{t3a} as well as
% the statement of
% Theorem \ref{t3c} are true for $\mu^{(k)\,\mrm{sym}}$ in place of $\mu^{(k)}$,
% except that $\log2$ and $\log3$ in (vi) and Theorem \ref{t3c}, respectively,
% should be replaced by $2\log2$ and $2\log3$.
% For the proof of the symmetrization versions, note the following.
% In (i) $\ds\limsup_{z\to\infty}|\wh\mu^{(0)\,\mrm{sym}}(z)|>0$ follows from
% $\ds\limsup_{z\to\infty}|\wh\mu^{(0)}(z)|>0$.  In (ii) and (iii), $\mu^{(0)\,\mrm{sym}}$
% is absolutely continuous with bounded continuous density if so is $\mu^{(0)}$.
% In (iv)--(vi) and in Theorem \ref{t3c} use Theorem \ref{t7d} (iii) and (iv).
The statement of Theorem \ref{t3c} is true for
$\mu^{(k)\,\mrm{sym}}$ in place of $\mu^{(k)}$, except that $\log3$
should be replaced by $2\log3$.

\bigskip

\noindent
Alexander Lindner\\
Institut f\"ur Mathematische Stochastik,
Technische Universit\"at Braunschweig, Pockels\-stra{\ss}e 14, D-38106 Braunschweig, Germany\\
email: a.lindner@tu-bs.de
\medskip

\noindent
Ken-iti Sato\\
Hachiman-yama 1101-5-103, Tenpaku-ku, Nagoya, 468-0074 Japan\\
email: ken-iti.sato@nifty.ne.jp
\end{document}